\documentclass [12pt]{article}
\usepackage {a4}
\usepackage{authblk}
\usepackage {amsfonts}
\usepackage {amssymb,amsmath,amsthm}
\usepackage [utf8]{inputenc}
\usepackage [english]{babel}
\usepackage{subcaption}
\usepackage{enumerate}
\usepackage{graphicx}
\usepackage{fullpage}
\usepackage{bbm}
\usepackage{hyperref}
\usepackage{todonotes}
\usepackage{verbatim}
\usepackage{afterpage}

\newcommand{\id}{\mathbbm{1}}
\newcommand{\R}{\mathbb{R}}

\newcommand{\Pp}{\mathbb{P}}
\newcommand{\DD}{\mathcal{D}}

\newcommand{\N}{\mathbb{N}}
\newcommand{\PP}{\mathcal{P}}
\newcommand{\MM}{\mathcal{M}}
\newcommand{\Zz}{\mathcal{Z}}
\newcommand{\EE}{\mathcal{E}}
\newcommand{\I}{\mathcal{I}}
\newcommand{\E}{\mathbb{E}}

\newcommand{\Z}{\mathbb{Z}}
\newcommand{\Ss}{\mathcal{S}}
\newcommand{\GSs}{\mathfrak{S}}

\newcommand{\sgn}{\mathrm{sgn}}

\newcommand{\Inv}{\mathrm{Inv}}
\newcommand{\Id}{\mathrm{id}}
\newcommand{\rev}{\mathrm{rev}}

\newtheorem{theorem}{Theorem}[section]
\newtheorem*{theorem*}{Theorem}

\newtheorem{lemma}[theorem]{Lemma}
\newtheorem{corollary}[theorem]{Corollary}
\newtheorem{proposition}[theorem]{Proposition}

\theoremstyle{definition}
\newtheorem{remark}[theorem]{Remark}

\begin{document}

\selectlanguage{english}

\title{The global and local limit of the continuous-time Mallows process}

\author[$\star$]{Radosław Adamczak}
\author[$\star$]{Michał Kotowski}

\affil[$\star$]{Institute of Mathematics, University of Warsaw}

\maketitle

\begin{abstract}
Continuous-time Mallows processes are processes of random permutations of the set $\{1, \ldots, n\}$ whose marginal at time $t$ is the Mallows distribution with parameter $t$. Recently Corsini showed that there exists a unique Markov Mallows process whose left inversions are
independent counting processes. We prove that this process admits a global and a local limit as $n \to \infty$. The global limit, obtained after suitably rescaling space and time, is an explicit stochastic process on $[0,1]$ whose description is based on the permuton limit of the Mallows distribution, analyzed by Starr. The local limit is a process of permutations of $\mathbb{Z}$ which is closely related to the construction of the Mallows distribution on permutations of $\mathbb{Z}$ due to Gnedin and Olshanski. Our results demonstrate an analogy between the asymptotic behavior of Mallows processes and the recently studied limiting properties of random sorting networks.
\end{abstract}

\tableofcontents

\begin{section}{Introduction}\label{sec:introduction}

The topic of this paper are limits of permutation-valued stochastic processes related to the Mallows model of random permutations. Let $\Ss_n$ be the group of permutations of $[n] = \{1, \ldots, n\}$. For $q \in [0, \infty)$ and $n \in \N$ the \emph{Mallows distribution with parameter $q$} is the probability measure $\pi_{n,q}$ on $\Ss_n$ defined by
\begin{equation}\label{eq:definition-of-mallows}
\pi_{n,q}(\sigma) = \frac{q^{\Inv(\sigma)}}{\Zz_{n,q}},
\end{equation}
where $\Inv(\sigma) = \left|\left\{ (i,j) \, : \, i < j, \sigma(i) > \sigma(j) \right\}\right|$ is the number of inversions of $\sigma$ and $\Zz_{n,q} = \prod\limits_{i=1}^{n} (1 + q + \ldots + q^{i-1}) $ is the normalizing constant.

For $q = 0$ the measure is concentrated on the identity permutation, while $q=1$ corresponds to the uniform distribution on $\Ss_n$. Informally speaking, the Mallows distribution for $q \in [0,1)$ tends to favor more ordered permutations, closer to the identity permutation $\Id_n = (1 \, 2 \, \ldots \, n)$, while for $q \in (1, \infty)$ it favors permutations closer to the reverse permutation $\rev_n = (n \, \ldots \, 2 \, 1)$.

This model, introduced originally by Mallows (\cite{mallows}) in statistical ranking theory, has attracted considerable attention in a number of contexts, including the study of its cycle structure (\cite{gladkich-peled}, \cite{he-muller-verstraaten}), asymptotics of the longest increasing subsequence (\cite{bhatnagar-peled}, \cite{basu-bhatnagar}, \cite{mueller-starr}), Markov chain mixing times (\cite{diaconis-ram}, \cite{BBHM}), integrable systems in statistical physics (\cite{starr}), binary search trees (\cite{addario-berry-corsini}, \cite{corsini-trees}) and stable matchings (\cite{stable-matchings}). We would also like to highlight its role in recent investigations of interacting particle systems such as ASEP and its generalizations, see, e.g. \cite{bufetov-hecke}, \cite{bufetov-nejjar}, \cite{bufetov-chen}. More general exponential families of permutations and their significance in statistics are discussed in \cite{MR3476619}.

To study the asymptotic behavior of Mallows random permutations it is natural to consider their convergence to an appropriately defined limiting object as $n \to \infty$. One such notion of a limit (a `global' one) is provided by the theory of permutons (\cite{permuton-paper}), by now a well-developed subject. These are measures on the unit square with uniform marginals, obtained by looking at the permutation matrix of a permutation `from far away' (which corresponds to scaling spacing by $\frac{1}{n}$). Starr (\cite{starr}) and later Starr and Walters (\cite{starr-walters}) proved that Mallows permutations sampled from $\pi_{n,q}$ admit a deterministic limit in this sense when $q = q_{n} \to 1$ is scaled appropriately as $n \to \infty$ (more precisely, $q_n \approx 1 + \frac{\beta}{n}$ for $\beta \in \R$).

Another kind of a limit, a `local' one, is obtained when one does not rescale space and instead considers permutations (bijections) of $\Z$ as limiting objects. This notion has its roots in the work of Gnedin and Olshanski (\cite{gnedin-olshanskii-1}, \cite{gnedin-olshanskii}), who proved the existence of the analog of the Mallows distribution for permutations of $\Z$ for $q \in (0,1)$.

The focus of this work are similar kinds of limits for permutation-valued \emph{stochastic processes}. Such processes have been of interest both from the point of view of probability theory and statistical physics (\cite{Goldschmidt:2011aa}). Perhaps the most well-known example is the interchange process, obtained by composing transpositions along edges of a given finite or infinite graph. We refer the reader to \cite{MR3726904} for some background on this process. Another example, which will be particularly relevant for us as a point of comparison, is given by \emph{random sorting networks} (\cite{random-sorting-networks}). These are shortest paths in $\Ss_n$ from the identity permutation to the reverse permutation, chosen uniformly at random, when adjacent transpositions are used as generators. Their remarkable structure was studied in a number of works, culminating in the proof of the sine curve conjecture (\cite{MR4467309}) -- after appropriate space and time rescaling, elements in a random sorting network follow sine curves described by the so-called Archimedean process. The process also possesses a local limit (\cite{random-sorting-networks-local}, \cite{gorin-rahman}) which is a stationary process of permutations of $\Z$, with interesting relations to random matrix theory.

We will be interested in continuous-time \emph{Mallows processes}, which are processes of permutations whose marginal at time $t$ is given by the Mallows distribution with parameter $t$. Thus, the parameter $q$ in the definition of $\pi_{n,q}$ plays the role of time. Such processes were introduced recently in \cite{corsini}. As $q=0$ corresponds to the identity permutation and for $q \to \infty$ the distribution $\pi_{n,q}$ is concentrated close to the reverse permutation, Mallows processes can be thought of as interpolating between the identity and the reverse permutation in $\Ss_n$, not unlike random sorting networks.

Indeed, we will prove that under certain assumptions a Mallows process, like a random sorting network, admits both a global limit (in which path of an element is a stochastic process with values in $[0,1]$) and a local limit (which is a process of permutations of $\Z$). In the global limit, with high probability elements follow random curves given by an explicit formula, like in the sine curve conjecture. The local limit is a transposition process which admits a description in terms of a family of Markov birth processes. This is the content of our main theorems, Theorem \ref{th:main-theorem-global} and Theorem \ref{th:main-theorem-local}, which we state below after providing necessary preliminary definitions.

\begin{subsection}{Preliminaries}

\paragraph{The Mallows distribution and inversions numbers.} It will be convenient to work with the following description of permutations. For a permutation $\sigma \in \Ss_n$ we define its \emph{left inversion vector} as the sequence $(\ell_1(\sigma), \ldots, \ell_n(\sigma))$, where
\[
\ell_i(\sigma) = \left|\left\{ j \in \{0, \ldots, i-1\} \, : \, \sigma(j) > \sigma(i) \right\}\right|
\]
is the number of \emph{left inversions} of element $i$ in permutation $\sigma$. It is classical that admissible inversion vectors of length $n$ are in bijection with permutations on $n$ elements. Given a vector $\left( \ell_1, \ldots, \ell_n \right) \in \{0\} \times \{0,1\} \times \ldots \times \{0, 1, \ldots, n-1\}$, the corresponding permutation $\sigma$ is constructed by setting $\sigma(n) = n - \ell_n$ and then recursively $\sigma(k) = x_{k - \ell_k}$, where $(x_1, \ldots, x_k)$ are elements of $[n] \setminus \{\sigma(n), \ldots, \sigma(k+1)\}]$ sorted in ascending order.

An important property of the Mallows distribution is that for $\sigma \sim \pi_{n,q}$ the inversion numbers $\ell_{i}(\sigma)$ are independent, with $\ell_{i}(\sigma)$ following the distribution
\[
\Pp\left( \ell_{i}(\sigma) = j \right) = \frac{q^{j}}{\sum\limits_{k=0}^{i-1} q^k}, \, j = 0, \ldots, i - 1.
\]
For $q < 1$ this is the truncated geometric distribution with parameter $1-q$. Thus the above bijection provides a simple way of sampling permutations from $\pi_{n,q}$.
\paragraph{Continuous-time Mallows processes.} Let $\sigma^n = \left(\sigma_{t}^{n}\right)_{t \in [0, \infty)}$ be a continuous-time stochastic process with values in $\Ss_n$. We will call $\sigma^n$ a \emph{continuous-time Mallows process} (or simply a Mallows process) if it is a c\`{a}dl\`{a}g stochastic process such that for all $t \in [0, \infty)$ the distribution of $\sigma^{n}_{t}$ is $\pi_{n,t}$. Such processes were first considered in \cite{corsini}, where the reader can find a more general discussion.

In this paper we will be interested in a specific Mallows process possessing certain independence and Markovianity properties.  Following \cite{corsini} we call a Mallows process $\left(\sigma_{t}^{n}\right)_{t \in [0, \infty)}$ \emph{regular} if it satisfies the following three conditions:
\begin{itemize}
\item for any $i \in [n]$ and $t \leq t'$ we have $\ell_i\left(\sigma_{t}^{n}\right) \leq \ell_i\left(\sigma_{t'}^{n}\right)$,
\item $\Inv(\sigma_{t}^{n}) \leq \Inv(\sigma_{t-}^{n}) + 1$,
\item the processes $\left( \ell_i \left(\sigma_{t}^{n}\right) \right)_{t \in [0, \infty)}$ are independent over $i \in [n]$.
\end{itemize}

Consider now for any $i \in [n]$ a time-inhomogeneous Markov birth process $(\ell^{n}_{i}(t))_{t \in [0, \infty)}$ with values in ${0, \ldots, i-1}$, whose rate for jumping from $j$ to $j+1$ at time $t$ is given by
\begin{equation}\label{eq:def-of-rates}
\begin{cases}
p_{i}(j,t) = \frac{1}{1-t} \left( j+1 - \frac{i t^{i-j-1}(t^{j+1} - 1)}{t^i - 1} \right), & \, j = 0, \ldots, i-1, \, t \neq 1, \\
p_{i}(j,1) = \frac{1}{2}(j+1) (i-j-1),  & \, t = 1.
\end{cases}
\end{equation}
We now let $\sigma^n = (\sigma_{t}^{n})_{t \in [0, \infty)}$ be the \emph{birth Mallows process} obtained by considering \emph{independent} processes $\left(\ell^{n}_i(t)\right)_{t \in [0, \infty)}$, $i=1, \ldots, n$, and defining $\sigma^{n}_{t}$ to be the unique permutation $\sigma$ with inversion vector $(\ell_1(\sigma), \ldots, \ell_n(\sigma)) = \left( \ell^{n}_{1}(t), \ldots, \ell^{n}_{n}(t) \right)$. The definition is correct thanks to the bijection between inversion vectors and permutations. It was shown in \cite{corsini} that $\sigma^n$ is indeed a Mallows process, i.e., $\sigma_{t}^{n} \sim \pi_{n,t}$ for each $t \in [0, \infty)$.

Furthermore, we have the following uniqueness result, which serves as an additional motivation to study the birth Mallows process -- the birth Mallows process is the unique regular Mallows process which is also Markov (\cite[Theorem 1.1]{corsini}). Note that in \cite{corsini} the formula for the rate $p_{i}(j,t)$ was given in terms of a certain sum, but the expression actually simplifies to the one given in \eqref{eq:def-of-rates}. From now  on whenever we speak of `the Mallows process' we will mean the birth Mallows process $\sigma^n$ defined above.

\end{subsection}

\begin{subsection}{Main results}

We are now in a position to formulate the main results of our paper, describing two limits of the Mallows process as $n \to \infty$ -- the global one (in which we are interested in statistical behavior of elements in macroscopic neighborhoods) and the local one (in which we look at the asymptotic behavior of individual elements).

Before that, let us comment on the necessity (or lack thereof) of rescaling time in order to obtain a nontrivial limiting process. Let us consider $q \leq 1$ (the argument for $q \geq 1$ is symmetric). It is known (see, e.g., \cite[Theorem 1.1]{bhatnagar-peled}) that for $\sigma \sim \pi_{n,q}$ the probability of the displacement $|\sigma(i) - i|$ being at least $k$ is roughly $q^k$. Therefore, all the displacements will be $o(n)$ (and thus vanishing after rescaling space) unless $q = q_n \approx 1 - \frac{\beta}{n}$ for some $\beta \geq 0$. This is consistent with the scaling for the permuton limit from \cite{starr} mentioned in the introduction. Thus to see nontrivial displacements in the global limit we need to zoom into a window of size $\sim \frac{1}{n}$ around time $t=1$.

Likewise, since in the local limit we would like to observe asymptotic trajectories of individual elements, we cannot expect to obtain a process defined up to and past $t=1$. As the distribution $\pi_{n,q}$ for $q=1$ is the uniform distribution, the limiting process would need to have as its marginal at time $1$ a uniform distribution on permutations of $\Z$, which does not exist. This corresponds to the number of jumps of each element in the limiting process exhibiting an explosion as $t \to 1$.

\paragraph{Global limit.} To state our first result, the convergence of the Mallows process to a global limit, we introduce the following rescaling of time and space. For $i \in [n]$ let $X_{i}^{n}(t) = \frac{1}{n} \sigma_{e^{t/n}}^{n}(i)$ denote the rescaled (global) trajectory of the element $i$ in the Mallows process after $t \mapsto e^{t / n}$ time rescaling. Thus each $\left( X_{i}^{n}(t)\right)_{t \in [0,\infty)}$ is a random path with values in $[0,1]$. Let

\begin{equation}\label{eq:formula-for-global-limit}
z_{x,a}(t) = \frac{1}{t} \log \left( \frac{e^{xt}(1-a) + ae^t}{e^{xt}(1-a) + a} \right),
\end{equation}
where $x,a \in [0,1]$, $t \in \R$ and we take $z_{x,a}(0) = a$. Note that $\lim\limits_{t \to -\infty} z_{x,a}(t) = x$ and $\lim\limits_{t \to \infty} z_{x,a}(t) = 1-x$.

Our first result is
\begin{theorem}\label{th:main-theorem-global}
For any $T > 0$, $\alpha > 0$ and $\varepsilon > 0$ we have the following convergence in probability
\[
\Pp\left( \max\limits_{i \in (\alpha n, (1-\alpha)n)} \sup\limits_{t \in [-T,T]} \left| X_{i}^{n}(t) - z_{\frac{i}{n}, X_{i}^{n}(0)}(t) \right| > \varepsilon \right) \xrightarrow{n \to \infty} 0,
\]
where $z_{x,a}(t)$ is given by \eqref{eq:formula-for-global-limit}.
\end{theorem}

The theorem can be interpreted in the following way -- apart possibly from elements which start close to the boundary points $0$ or $1$, with high probability each element $x = \frac{i}{n}$ is assigned at random its desired position $a = X_{i}^{n}(0)$ at time $0$ and then approximately follows a deterministic trajectory $z_{x,a}(t)$ depending only on $x$ and $a$. In the original Mallows process without the $t \mapsto e^{t/n}$ time rescaling this would correspond to starting at position $x$ at time $0$ and reaching position $a$ at time $1$. Results of simulations illustrating the behavior of the Mallows process for $n=200$ along with the limiting trajectories $z_{x,a}$ are shown in Figure \ref{fig:graph1} and Figure \ref{fig:graph2}.

In fact, we expect the theorem to hold with maximum taken over all elements $i \in [n]$, not necessarily bounded away from $1$ and $n$, but our proof stops short of showing that.

\begin{figure}[!h]
\captionsetup[subfigure]{labelformat=empty}
	\centering
	\begin{subfigure}[t]{0.4\textwidth}
	    \includegraphics[width=\linewidth]{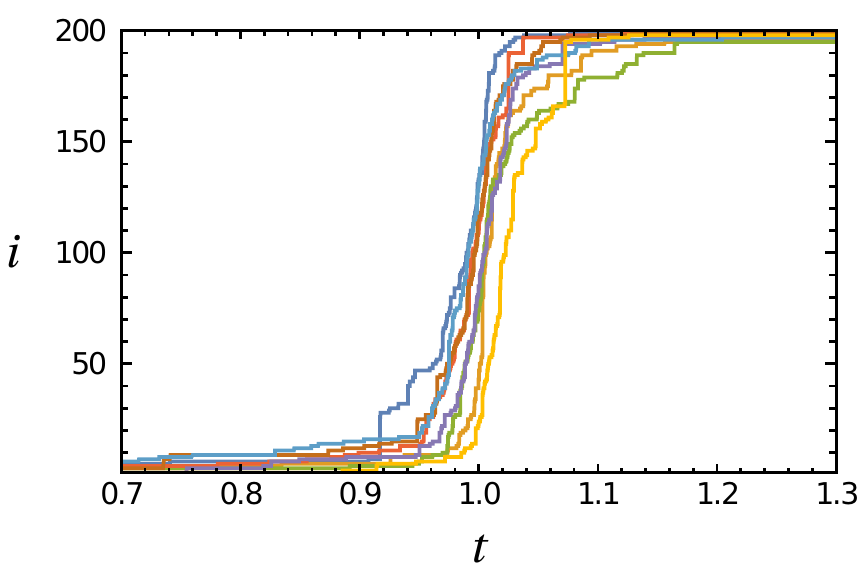}
        \subcaption{$i=1$}
	\end{subfigure}
	\hspace{1pc}
    \vspace{0.5pc}
	\begin{subfigure}[t]{0.4\textwidth}
	    \includegraphics[width=\linewidth]{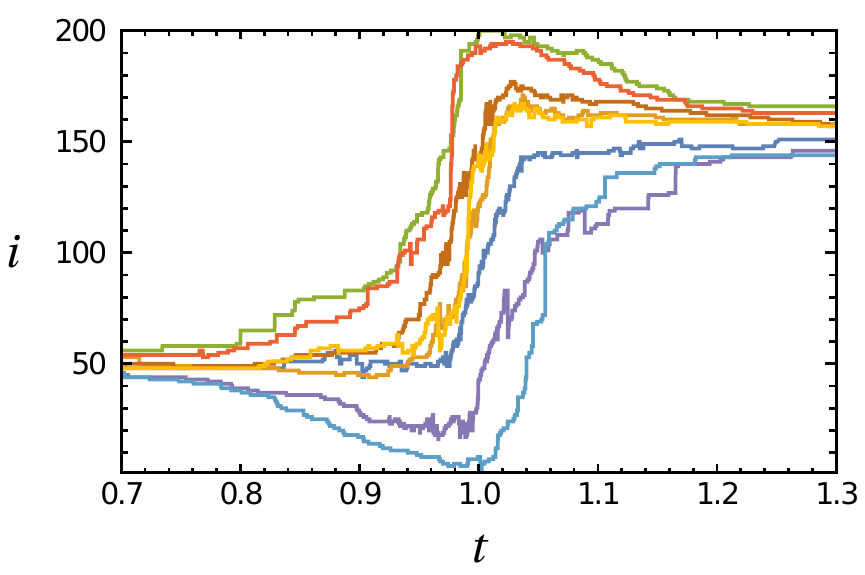}
        \subcaption{$i=50$}
	\end{subfigure}
	\begin{subfigure}[t]{0.4\textwidth}
	    \includegraphics[width=\linewidth]{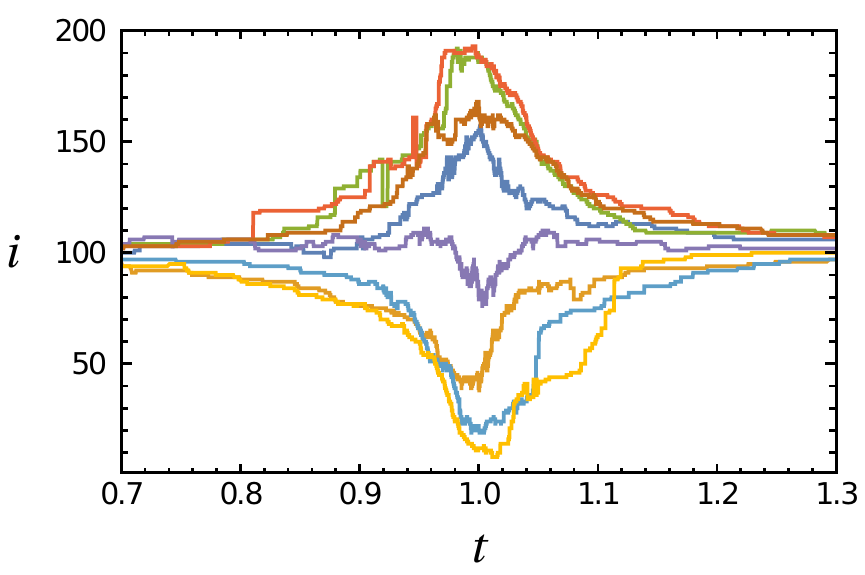}
        \subcaption{$i=100$}
	\end{subfigure}
	\hspace{1pc}
    \vspace{0.5pc}
	\begin{subfigure}[t]{0.4\textwidth}
	    \includegraphics[width=\linewidth]{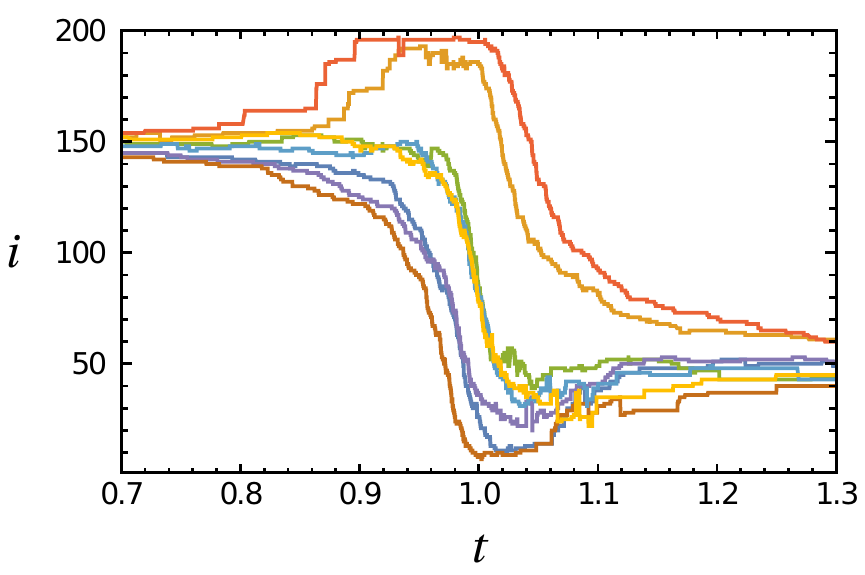}
        \subcaption{$i=150$}
	\end{subfigure}
    \begin{subfigure}[t]{0.4\textwidth}
	    \includegraphics[width=\linewidth]{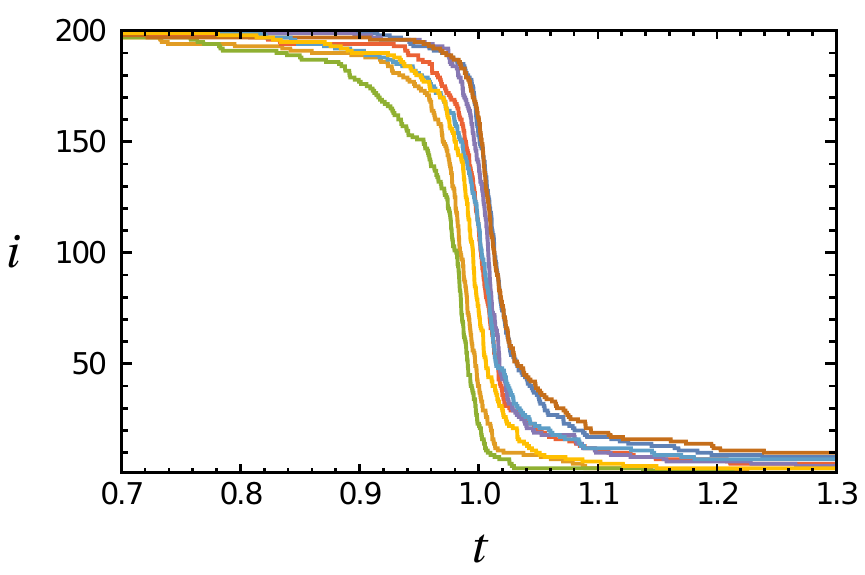}
        \subcaption{$i=200$}
	\end{subfigure}
	\caption{Trajectories $\sigma^{n}_{t}(i)$, with $n=200$, of elements $i=1,50,100,150,200$ for $t \in [0.7,1.3]$ across different realizations of the process.}
\label{fig:graph1}
\end{figure}

\begin{figure}[!h]
\captionsetup[subfigure]{labelformat=empty}
	\centering
	\begin{subfigure}[t]{0.4\textwidth}
	    \includegraphics[width=\linewidth]{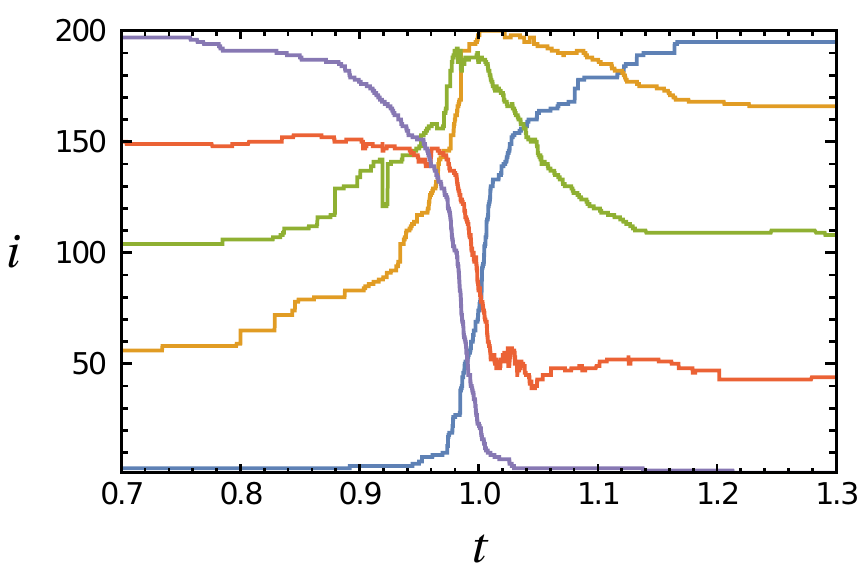}
	\end{subfigure}
    \hspace{1cm}
    \begin{subfigure}[t]{0.4\textwidth}
	    \includegraphics[width=\linewidth]{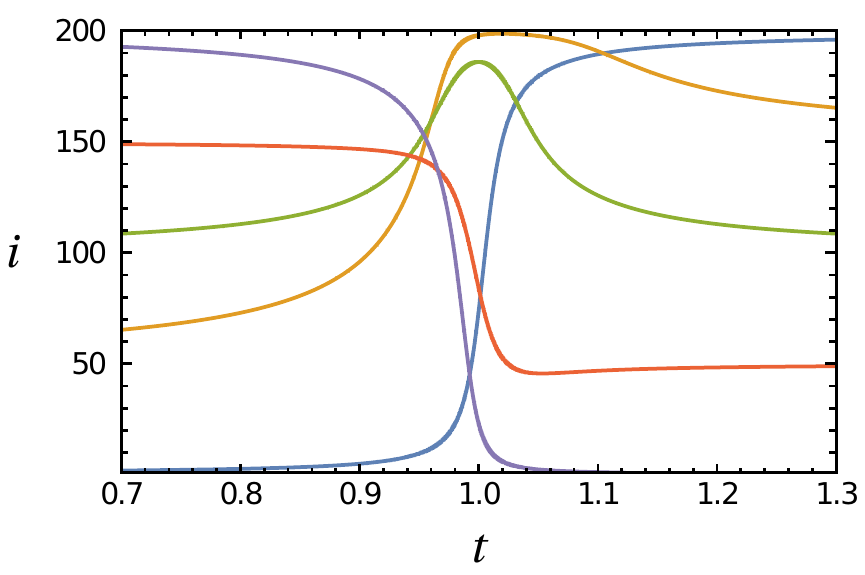}
	\end{subfigure}
	\caption{Trajectories $\sigma^{n}_{t}(i)$, with $n=200$, of elements $i=1,50,100,150,200$ for $t \in [0.7,1.3]$ in the same realization of the process (left) together with the corresponding limiting trajectories $z_{x,a}$ (right)}
\label{fig:graph2}
\end{figure}

Theorem \ref{th:main-theorem-global} implies the following corollary about the behavior of a randomly chosen element. Let $\left(Y^{n}\right)_{t \in [-T,T]}$ be the \emph{random particle process} defined by choosing $i \in [n]$ uniformly at random and then following the trajectory $\left(X_{i}^{n}(t)\right)_{t \in [-T,T]}$. Note that for $t=0$ we have $\sigma_{e^{t/n}}^{n} = \sigma_{1}^{n}$, which is uniformly distributed on $\Ss_n$ (as it corresponds to the Mallows distribution with parameter $1$), so each $X_{i}^{n}(0)$ has uniform distribution on $\{1/n,2/n,\ldots,n/n\}$. Since $\alpha$ in Theorem \ref{th:main-theorem-global} above can be arbitrarily small, we obtain the following

\begin{corollary}\label{cor:main-result-random-particle}
For any $T > 0$ let $Y = \left(Y^{n}\right)_{t \in [-T,T]}$ be the random particle process, regarded as a random path in the Skorokhod space $D\left([-T,T], [0,1]\right)$. Then we have the following convergence in distribution as $n \to \infty$:
\[
Y^{n} \stackrel{d}{\rightarrow} Z,
\]
where $Z = \left(z_{X,A}(t)\right)_{t \in [-T,T]}$ with $X$, $A$ independent and distributed uniformly on $[0,1]$.
\end{corollary}

\paragraph{Local limit.} We now state the result about the second kind of limit of the Mallows process, the local limit. Here, for reasons that will be apparent in a moment, we need to restrict our attention to times $t \in [0,1)$.

Let $\GSs$ denote the group of all permutations of $\Z$, i.e., all bijections of $\Z$ with itself. A permutation $\sigma \in \GSs$ will be called \emph{balanced} if
\[
\left| \left\{ i \leq 0 \, : \, \sigma(i) \geq 1 \right\} \right| = \left| \left\{ i \geq 1 \, : \, \sigma(i) \leq 0 \right\} \right| < \infty.
\]
The set of all balanced permutations will be denoted by $\GSs^{bal}$.

For $q \in [0,1)$ Gnedin and Olshanski (\cite{gnedin-olshanskii}) constructed a measure $\Pi_q$ on $\GSs^{bal}$ which serves as an analog of the Mallows distribution $\pi_{n,q}$ on $\Ss_n$. This measure, whose existence underlies our local limit results, is discussed in more detail in Section \ref{sec:local-limit}. By a slight abuse of language we will call $\Pi_q$ the \emph{Mallows measure on $\Z$ with parameter $q$}.

An $\GSs$-valued path $\left( \Sigma_t \right)_{t \in [0,1)}$ will be called a \emph{permutation path} if for every $i \in \Z$ the trajectory $\left( \Sigma_{t}(i) \right)_{t \in [0,1)}$ is c\`{a}dl\`{a}g. We say that a sequence of permutation paths $\left( \Sigma^{n}_{t} \right)_{t \in [0,1)}$ converges to a permutation path $\left( \Sigma_{t} \right)_{t \in [0,1)}$ if for every $i \in \Z$ we have $\left( \Sigma^{n}_{t}(i) \right)_{t \in [0,1)} \to \left( \Sigma_{t}(i) \right)_{t \in [0,1)}$ as $n \to \infty$, where the latter convergence is convergence of paths in the Skorokhod space. A random permutation path will be called a \emph{permutation process} on $\Z$.

Let $\DD = D([0,1),\GSs^{bal})$ be the space of all c\`{a}dl\`{a}g functions on $[0,1)$ with values in $\GSs^{bal}$. In other words, $\DD$ consists of permutation paths whose values are balanced permutations. We provide the details concerning this space in the Appendix. Note that convergence in $\DD$ implies convergence of permutation paths in the sense introduced above.

Given the Mallows process $\left( \sigma^{n}_t \right)_{t \in [0,1)}$ and any sequence $k_n \in \Z$, we define the following shifted process
\begin{equation}\label{eq:shifted-process-n}
\Sigma^{n}_{t}(i) = \sigma^{n}_{t}\left( k_n + i \right) - k_n,
\end{equation}
where $i \in \Z$, $t \in [0,1)$ and we take $\Sigma^{n}_{t}(i) = i$ if $k_n + i \notin [n]$. In this way $\left( \Sigma^{n}_{t} \right)_{t \in [0,1)}$ can be considered as a permutation process on $\Z$ for any value of $n$ (in fact, a process of balanced permutations).

Our main result here is the following
\begin{theorem}\label{th:main-theorem-local}
There exists a permutation process $\left( \Sigma_{t} \right)_{t \in [0,1)}$ such that for any sequence $k_n \in \N$ with $k_n,n-k_n \to \infty$ as $n \to \infty$, we have the following convergence in distribution
\[
\left( \Sigma^{n}_{t} \right)_{t \in [0,1)} \stackrel{d}{\rightarrow} \left( \Sigma_{t} \right)_{t \in [0,1)}
\]
in the Skorokhod topology on $\DD$. The marginal distribution of $\Sigma_t$ at time $t$ is given by the Mallows measure $\Pi_t$ on $\Z$ with parameter $t$.
\end{theorem}

\begin{remark}\label{re:total-variation}
In fact our proof implies that for any $T \in [0,1)$ and $N \in \N$, the process $(\Sigma^n_t(i))_{i \in \{-N,\ldots,N\},t\in [0,T]}$ seen as an element of $D([0,T],\Z^{2N+1})$ converges in the total variation distance to $(\Sigma_t(i))_{i \in \{-N,\ldots,N\},t\in [0,T]}$.
\end{remark}

The intuitive picture behind Theorem \ref{th:main-theorem-local} is as follows. We fix a spatial location $x \in (0,1)$ together with a sequence $k_n$ satisfying $\frac{k_n}{n} \to x$ (this is a special case of the theorem) and observe the behavior of the Mallows process $\sigma^{n}_t$ in a window of constant size around $\lfloor xn \rfloor$, and then take $n \to \infty$. Note in particular that the distribution of the limiting process does not depend on $x$, i.e., is the same for all locations as long as we are not close to the boundary. For locations near $0$ or $1$, i.e., when either $k_n$ or $n - k_n$ does not go to infinity, one should consider processes of permutations of $\N$ rather than of $\Z$ and change all the statements accordingly, but this case is actually easier to handle. We discuss this briefly in Section \ref{sec:local-limit}.

The limiting process $\left(\Sigma_{t}\right)_{t \in [0,1)}$ admits a description in terms of a family of independent Markov jump processes and it is a transposition process. This is discussed in more detail in Section \ref{sec:local-limit}.

The fact that $\Sigma_t$ is distributed according to the Mallows measure on $\Z$ with parameter $t$ explains why we have to limit ourselves to times $t \in [0,1)$. Indeed, the Mallows measure on $\Z$ with parameter $t=1$ would correspond to the uniform distribution over permutations of $\Z$, which does not exist. This is reflected in the fact that the processes of inversions used to define $\left(\Sigma_t\right)_{t \in [0,1)}$ explode at $t=1$, as mentioned in earlier.

\end{subsection}

\begin{subsection}{Overview of the proofs and further questions}

\paragraph{Overview of the proofs.} We briefly describe the techniques used in the proofs of our main results.

To prove Theorem \ref{th:main-theorem-global}, we first derive the so-called fluid limit (see \cite{norris}) of the inversion processes $\ell_{i}(\sigma_{e^{t/n}}^{n})$, which consists in approximating the trajectory of a Markov process by a solution to a corresponding ODE. In particular, the inversion number $\ell_{i}(\sigma_{e^{t/n}}^{n})$ at any time $t$ is essentially determined by the inversion number at time $t=0$. Next, as the permuton limit of the Mallows permutations is deterministic and explicitly known (thanks to \cite{starr}), for a given element with initial position $x$ its inversion number and its position at time $t=0$ determine each other up to small error. The same is true for any other time $t$ and since the inversion process is approximated by an ODE, this proves that the trajectory $X_{i}^{n}(t)$ itself is almost deterministic. Along the way we employ a large deviation principle satisfied by Mallows permutations, which enables us to work separately with each time point from a sufficiently fine mesh of $[0,T]$. At the end we pass from $[0,T]$ to $[-T,T]$ by a symmetry argument.

As for Theorem \ref{th:main-theorem-local}, we first construct the limiting process $\Sigma_t$ in terms of a family of independent (time-inhomogeneous) birth processes $\ell_i(t)$, representing the left inversion numbers. This is analogous to what is done for the Mallows process $\sigma^{n}_{t}$ on $\Ss_n$, but more involved, since reconstructing a permutation of $\Z$ from its inversion numbers is not straightforward. Then we construct coupled copies of all the processes $\sigma^{n}_{t}$ and the limiting process by means of a dependent thinning procedure, in which successive jumps of the $\ell_i(t)$'s are rejected with appropriate probability. As the jump rates given by \eqref{eq:def-of-rates} approximate the jump rates of inversions of the limiting process, we are able to show that almost surely on any fixed interval $\{-N, \ldots, N\}$ the processes $\Sigma^{n}_{t}$ and $\Sigma_t$ agree for large $n$ enough.

\paragraph{Further questions.} We end the introductory section with a few questions and ideas for further research which might lead to additional insight into the Mallows process and permutation-valued processes in general.
\begin{itemize}
\item The Archimedean process, which appears as the limit of a random sorting network (\cite{MR4467309}), admits an important variational characterization. This was actually noticed by Brenier (\cite{brenier}) in the study of incompressible flows long before the random sorting network conjecture was formulated. Let $(X_t)_{t \in [0,1]}$ be a continuous process with values in $[0,1]$ whose marginal $X_t$ for any $t$ is uniform on $[0,1]$. Such processes are known as permuton processes (see \cite{balint-mustazee}) and they arise naturally as limits of permutation sequences. For a permuton process $X$ we define its energy as
\[
\mathcal{E}(X) = \frac{1}{2} \E \int\limits_{0}^{1} |\dot{X}_{t}|^2 \, dt.
\]
Brenier proved (see also \cite{balint-mustazee}) that the Archimedean process is the unique minimizer of energy among all permuton processes satisfying $X_0 = 1 - X_1$. This constraint is analogous to the requirement that a permutation sequence connect the identity permutation to the reverse permutation. Since in the Archimedean process the trajectories are sine curves, they satisfy the harmonic oscillator equation
\[
\ddot{X}_t = - X_t.
\]
This is the Euler-Lagrange equation for the action $X \mapsto \int\limits_{0}^{1} \left( \frac{1}{2}|\dot{X}_t|^2 - p(X_t) \right) dt$, where $p(x) = \frac{1}{2} x^2$ is the pressure function.

Is it possible to find a similar variational characterization for the global limit of the Mallows process from Theorem \ref{th:main-theorem-global}? Since the limiting trajectories $z_{x,a}(t)$ depend on two parameters, this strongly suggests that they should satisfy a second-order Euler-Lagrange equation like the trajectories of the Archimedean process, perhaps with the action being explicitly time-dependent.

\item To understand the asymptotic behavior of processes of permutations of $\Z$, it is natural to study the \emph{speed distribution} of elements. For a variety of interacting particle systems, such as TASEP (\cite{tasep}), ASEP (\cite{aggarwal}) or the local limit of a random sorting network (\cite{random-sorting-networks-local}), it is known that if $X_{t}(i)$ is the position of a given particle at time $i$, then the limit
\[
\lim\limits_{t \to \infty} \frac{X_{t}(i) - i}{t}
\]
exists almost surely and one can describe the limiting distribution. Thus one can think of $i$ as `choosing' its speed at random and moving along the associated trajectory as $t \to \infty$. Does a similar limit exist for the local limit of the Mallows process from Theorem \ref{th:main-theorem-local}? Here one should consider $t \to 1$ and the known results about the typical displacement $\sigma_{t}(i) - i$ (see, e.g., \cite[Theorem 1.1]{bhatnagar-peled} and \cite[Remark 5.2]{gnedin-olshanskii}) suggest that the right normalization is $\frac{1}{1-t}$ in the denominator. We note that \cite[Theorem 5.1]{gnedin-olshanskii} provides a series representation for the time marginal of the displacement, which hints at a possible approach to work out the limiting speed distribution.

\item It would be interesting to obtain a more detailed description of the jump distribution of the local limit of the Mallows process. For example, what is the distribution of the displacement $\Sigma_t(i) - i$ at the first time $i$ is transposed with some other element? More generally, is it possible to describe the transpositions of $\Sigma_t$ by means of an explicit point process, as was done for the local limit of random sorting networks in \cite{gorin-rahman}? Some results on the distribution of early jump times for Mallows processes on $\Ss_n$ were proved in \cite{corsini}.
\item In this work, following \cite{corsini}, we have defined the Mallows distribution $\pi_{n,q}$ on $\Ss_n$ in terms of left inversions $\ell_i$. We could just as well have chosen to work with right inversions $r_i$ (as, e.g., is done in \cite{gnedin-olshanskii} for permutations of $\Z$), since both choices yield the same Mallows distribution for fixed parameter $q$. Note, however, that the corresponding Mallows processes are not the same, as can be easily checked on explicit examples for small $n$. An interesting question would be to construct a Mallows process still satisfying $\Inv(\sigma_{t}^{n}) \leq \Inv(\sigma_{t'}^{n})$ for $t < t'$ and $\Inv(\sigma_{t}^{n}) \leq \Inv(\sigma_{t-}^{n}) + 1$, but which would be at the same time `canonical', in the sense of not distinguishing between left and right inversions in its definition.
\end{itemize}

\paragraph{Organization of the paper.} Section \ref{sec:global-limit} is devoted to the proof of Theorem \ref{th:main-theorem-global}, with Subsection \ref{subsec:permutons} introducing the permuton limits and the large deviation principle for Mallows permutations, Subsection \ref{subsec:inversions} deriving the ODE approximation for the evolution of inversions and Subsection \ref{subsec:global-proof} containing the main proof. In Section \ref{sec:local-limit} we prove Theorem \ref{th:main-theorem-local}. Subsection \ref{subsec:mallows-on-z} introduces the Mallows distribution on $\Z$ and Subsection \ref{subsec:construction-of-limit} contains the construction of the limiting process $\Sigma_t$. Finally, Subsection \ref{sec:proof-of-local-limit} contains the proof of the theorem itself. Several technical statements used in the proofs have been placed in the Appendix.
\end{subsection}
\end{section}

\begin{section}{Global limit}\label{sec:global-limit}

\begin{subsection}{Permutons and large deviations}\label{subsec:permutons}
\paragraph{Permuton limits.} We start by introducing necessary notions related to permutation limits. Let $\MM([0,1]^2)$ denote the space of Borel probability measures on $[0,1]^2$. For a permutation $\sigma \in \Ss_n$ its associated empirical measure is defined as
\[
\mu_{\sigma} = \frac{1}{n} \sum\limits_{i=1}^{n} \delta_{\left( \frac{i}{n}, \frac{\sigma(i)}{n} \right)}.
\]
A \emph{permuton} is a Borel probability measure on $[0,1]^2$ with uniform marginals on each coordinate. We say that a sequence of permutations $\sigma^{n} \in \Ss_n$ converges to a permuton $\mu$ if $\mu_{\sigma^{n}} \to \mu$ weakly. It is readily seen that any weak limit of a sequence $\mu_{\sigma^{n}}$ is necessarily a permuton and conversely, by general theory of permutation limits (\cite{permuton-paper}) any permuton can be realized as a limit of permutations in the above sense. The set of all permutons will be denoted by $\PP \subseteq \MM([0,1]^2)$. Note that $\PP$ is closed in the topology of weak convergence.

For permutons the weak topology is actually equivalent to a stronger topology, the \emph{box topology} (\cite[Lemma 5.3]{permuton-paper}), which will be useful later on. It is given by the following metric $d_{\Box}$ -- for permutons $\mu, \nu$ and random variables $(X_1, Y_1) \sim \mu$, $(X_2, Y_2) \sim \nu$ we define
\begin{equation}\label{eq:box-topology}
d_{\Box}(\mu, \nu) := \sup\limits_{\substack{a < b \in [0,1] \\ c < d \in [0,1]}} \Big| \Pp(X_1 \in [a,b], Y_1 \in [c,d]) - \Pp(X_2 \in [a,b], Y_2 \in [c,d])\Big|.
\end{equation}

For any $\beta \in \R$ let $\Pp_{n, \beta}(\cdot)$ denote the probability measure corresponding to sampling $\sigma^{n} \in \Ss_n$ from the Mallows distribution $\pi_{n,q}$ with parameter $q = e^{\beta/n}$. Thus under $\Pp_{n, \beta}$ the measure $\mu_{\sigma^{n}}$ is a random empirical measure. Starr (\cite{starr})  proved a law of large numbers for $\sigma^{n}$, which translated into the language of permuton limits reads as follows

\begin{theorem}[\cite{starr}, Theorem 1.1]\label{th:starr-lln}
Let $\beta \in \R$ and let $\sigma^{n}$ be sampled according to $\Pp_{n,\beta}$. We have the following convergence in probability:
\[
\mu_{\sigma^{n}} \xrightarrow{\Pp} \mu_{\beta},
\]
where the (deterministic) measure $\mu_{\beta}$ is given by the following density $\rho_{\beta}$ with respect to the Lebesgue measure on $[0,1]^2$:
\begin{align}\label{eq:definition-of-rho-beta}
\rho_{\beta}(x,y) = \frac{(\beta / 2) \sinh(\beta / 2)}{\left( e^{-\beta / 4} \cosh(\beta[x-y]/2) - e^{\beta / 4} \cosh(\beta[x+y-1]/2) \right)^2}.
\end{align}
\end{theorem}

Note that in \cite{starr} the result was formulated for parameter $1 - \frac{\beta}{n}$ instead of $e^{\beta / n}$, but this does not influence the proof nor result, and we have accounted for the change of $-\beta$ to $\beta$ in the formula above. The density $\rho_{\beta}$ is a continuous function bounded away from $0$ for fixed $\beta$ and for $\beta = 0$ we obtain $\rho_{0}(x,y) \equiv 1$ as expected.

The law of large numbers was later refined by Starr and Walters (\cite{starr-walters}) to a large deviation principle which we discuss below.
\paragraph{Large deviations.} For any $\beta \in \R$ let $\I_{\beta} : \MM([0,1]^2) \to [0,\infty]$ be defined in the following way -- we take $\I_{\beta}(\mu) = + \infty$ if $\mu$ is not a permuton, and otherwise
\begin{equation}\label{eq:rate-function}
\I_{\beta}(\mu) = S(\mu | \lambda^{\otimes 2}) + \beta \EE(\mu) + p(\beta),
\end{equation}
where
\begin{itemize}
\item $S(\mu | \nu) = \int\limits_{[0,1]^2} \log \left(\frac{d\mu}{d\nu}\right) d\mu$ is the relative entropy of $\mu$ with respect to $\nu$,
\item $\lambda^{\otimes 2}$ is the Lebesgue measure on $[0,1]^2$,
\item $\EE : \MM([0,1]^2) \to \R$ is defined as
\[
\EE(\mu) = \frac{1}{2} \int\limits_{[0,1]^2} \int\limits_{[0,1]^2} \id_{\{ (x_1 - x_2)(y_1 - y_2) < 0 \}} \, d\mu(x_1, y_1) \, d\mu(x_2, y_2),
\]
\item $p(\beta) = \int\limits_{0}^{1} \log \left( \frac{1 - e^{-\beta x}}{\beta x} \right) dx$, with $p(0) = 0$.
\end{itemize}
The function $\I_{\beta}(\cdot)$ is lower semi-continuous on $\MM([0,1]^2)$ and has compact sublevel sets (\cite[Lemma 3.1]{starr-walters}), i.e., it is a good rate function in the terminology of the theory of large deviations (we refer the reader to the monographs \cite{MR0793553,MR1619036}). Also, $\I_{\beta}(\mu)$ is continuous as a function of $\beta$ for fixed $\mu$.

The following large deviation principle was obtained by Starr and Walters (see also \cite{MR2391251,MR3476619}). The notions of open and closed sets in its formulation are understood in the sense of weak topology on the space of probability measures.

\begin{theorem}[\cite{starr-walters}, Proposition 3.2]\label{th:starr-walters-ldp}
Fix $\beta \in \R$ and let $\sigma^{n}$ be sampled according to $\Pp_{n,\beta}$. With the notation as above, for any closed set $A \subseteq \MM([0,1]^2)$ we have
\[
\limsup\limits_{n \to \infty} \frac{1}{n} \log \Pp_{n, \beta} \left( \mu_{\sigma^{n}} \in A \right) \leq - \inf\limits_{\mu \in A} \I_{\beta}(\mu)
\]
and for any open set $A \subseteq \MM([0,1]^2)$
\[
\liminf\limits_{n \to \infty} \frac{1}{n} \log \Pp_{n, \beta} \left( \mu_{\sigma^{n}} \in A \right) \geq - \inf\limits_{\mu \in A} \I_{\beta}(\mu).
\]
\end{theorem}
\begin{remark}
The large deviation result in \cite{starr-walters} is stated and proved for the `continuous' version of the Mallows model, which may be defined by replacing the empirical measure $\mu_{\sigma^{n}}$ by a permuton with density $f_{\sigma^{n}}(x,y) = n \id_{\sigma^{n}(\lceil nx\rceil) = \lceil ny\rceil}$. Nevertheless, the large deviation principle holds in the same form in both models, see \cite[Remark 3.4, Remark 3.5]{starr-walters} and Remark A.6 in version $4$ of \cite{starr-walters} on arXiv. 
\end{remark}

Let us remark that in our proofs we will only use the large deviation upper bound, exploiting the fact that the weak and box topology are equivalent on the set of permutons $\PP$.

A key fact is that the rate function $\I_{\beta}(\cdot)$ has a unique minimizer which is the measure $\mu_\beta$ from the Law of Large Numbers (Theorem \ref{th:starr-lln}).

\begin{proposition}[\cite{starr-walters}, Theorem 3.11]\label{prop:starr-walters-minimizer}
For each $\beta \in \R$ the measure $\mu_{\beta}$ satisfies $\I_{\beta}(\mu_{\beta}) = 0$ and is the unique minimizer of $I_\beta$ over the space $\MM([0,1]^2)$.
\end{proposition}

The large deviation principle has the following corollary, which will be important for the proof of Theorem \ref{th:main-theorem-global}. For a rectangle $R = [a,b] \times [c,d] \subseteq [0,1]^2$ and $\sigma \in \Ss_n$ let
\begin{align}\label{eq:definition-of-Delta}
\Delta_{R}(\sigma) = \frac{1}{n} \left| \left\{ i \in [n] \, \colon \,  \frac{i}{n} \in [a, b], \frac{\sigma(i)}{n} \in [c, d] \right\} \right|.
\end{align}
We have
\begin{corollary}\label{cor:ldp}
Fix $T > 0$, $\varepsilon > 0$ and $\beta \in [0,T]$. Let $\sigma^{n}$ be sampled according to $\Pp_{n,\beta}$. We have
\[
\Pp_{n, \beta} \Bigg( \sup\limits_{R} \bigg| \Delta_{R}\left(\sigma^{n}\right) - \int\limits_{R} \rho_{\beta}(x,y) \, dx \, dy \bigg| \geq \varepsilon \Bigg) \leq 2 e^{-cn},
\]
where the supremum is over all rectangles $R = [a,b] \times [c,d] \subseteq [0,1]^2$, $\rho_{\beta}$ is the density from Theorem \ref{th:starr-lln} and the constant $c = c_{\varepsilon, T} > 0$ depends only on $\varepsilon$ and $T$.
\end{corollary}

Since the proof is slightly technical, we defer it to the Appendix.
\end{subsection}
\begin{subsection}{Evolution of inversions}\label{subsec:inversions}

\paragraph{Approximation of rates.} Recall that in the Mallows process $\left(\sigma^{n}_{t}\right)_{t \in [0,\infty)}$ the left inversion vector at time $t$ is given by $\left( \ell_{1}\left(\sigma^{n}_{t}\right), \ldots, \ell_{n}\left(\sigma^{n}_{t}\right) \right) = \left( \ell^{n}_{1}(t), \ldots, \ell^{n}_{n}(t) \right)$, where each $\left( \ell^{n}_{i}(t) \right)_{t \in [0,\infty)}$ is a Markov birth process with birth rate given by \eqref{eq:def-of-rates}. In this subsection we show that under the global space and time scaling the evolution of $\ell^{n}_{i}(t)$ can be well-approximated by solutions of a deterministic ODE with random initial conditions.

Let $I_{i}(t) = \frac{1}{n} \ell^{n}_{i}( e^{t/n} )$, $i=1, \ldots, n$, denote the rescaled processes of left inversions (with the dependence on $n$ suppressed in the notation). For $x,y \in [0,1]$ and $t \in \R$ let
\begin{equation}\label{eq:rescaled-rates}
\begin{cases}
\lambda_{x}(y, t) = \frac{1}{t} \left( - y + \frac{x e^{t x}}{e^{t x} - 1}(1 - e^{-t y}) \right), & t \neq 0\\
\lambda_{x}(y, 0) = \frac{1}{2}y(x - y), & t = 0
\end{cases}
\end{equation}
where we set $\lambda_{0}(y, t) = \frac{1}{t}(-y + (1-e^{-ty})/t)$ for $t \neq 0$. Note that $\lambda_{x}(y, t)$ is continuous in $t$ and Lipschitz in $x$ and $y$, with Lipschitz constant uniform in $t \in [0,T]$ for fixed $T > 0$. For the sake of completeness we prove this as Lemma \ref{le:Lipschitz-condition} in the Appendix.

Let us remark that for simplicity of notation we have defined $\lambda_{x}(y,t)$ for all $x,y \in [0,1]$, but in fact we will only consider $y \leq x$ in the sequel.

The functions $\lambda_{x}$ approximate the rates $p_{i}$ from \eqref{eq:def-of-rates} as explained in the following lemma.

\begin{lemma}\label{le:rate-approximation} For any integers, $i,j$ such that $0\le j < i \le n$, any $x,y \in [0,1]$ and $t \in [0,T]$ we have
\begin{equation}\label{eq:intensity-approximation}
\left|\frac{1}{n^2} p_{i}\left(j, e^{\frac{t}{n}} \right) - \lambda_{x}(y, t) \right| \leq C_T\left(\left|\frac{i}{n} - x\right| + \left|\frac{j}{n} - y\right| + \frac{1}{n}\right),
\end{equation}
where $C_T$ depends only on $T$.
\end{lemma}

\begin{proof}
Observe that for $t > 0$,
\[
  \frac{1}{n^2} p_{i}\left(j, e^{\frac{t}{n}} \right) = \frac{t/n}{e^{t/n}-1} \cdot \lambda_{\frac{i}{n}}\left(\frac{j+1}{n},t\right)
\]
and so
\begin{displaymath}
  \left| \frac{1}{n^2} p_{i}\left(j, e^{\frac{t}{n}} \right) - \lambda_{x}(y, t)\right| \le \frac{t/n}{e^{t/n}-1} \left| \lambda_{\frac{i}{n}}\left(\frac{j+1}{n},t\right) - \lambda_x(y,t)\right|
  + \lambda_x(y,t)\left|\frac{t/n}{e^{t/n}-1} - 1\right|. 
\end{displaymath}
The second factor in the last term on the right-hand side is bounded by $Ct/n$ for some constant $C$, depending only on $T$. By Lemma \ref{le:Lipschitz-condition} from the Appendix the second factor in the first term is bounded by $C(\left|\frac{i}{n} - x\right| + \left|\frac{j+1}{n}-y\right|)$. These two observations also imply that the remaining factors are bounded by a constant, which proves the lemma for $t > 0$. The case of $t = 0$ is straightforward. 
\end{proof}

In particular, there exists $C = C_T > 0$, depending only on $T$, such that $p_{i}\left(j, e^{\frac{t}{n}}\right) \leq C n^2$ for any $i,j=1,\ldots,n$, $t \in [0,T]$, so that each $\ell^{n}_{i}(t)$ is a counting process with intensity bounded by $Cn^2$ (we recall the definition of a counting process and its intensity briefly in the Appendix).

\paragraph{The ODE limit.} We will now show the following ODE approximation for the joint dynamics of $I_{i}(t)$.
\begin{proposition}\label{prop:evolution-of-inversions}
Let $(I_{1}(t), \ldots, I_{n}(t))$ be as above and for $i=1,\ldots,n$ let $y_{i}(t)$, $t \in [0,T]$, denote the solution of the ODE
\[
\begin{cases}
& \dot{y}_{i}(t) = \lambda_{\frac{i}{n}}(y_{i}(t),t), \\
& y_{i}(0) = c_i
\end{cases}
\]
for some random variables $c_i \in [0,1]$. We have
\[
\Pp\left( \forall \, i \in [n] \, \sup\limits_{t \in [0,T]} \left| I_{i}\left( t \right) - y_{i}(t) \right| < C \left(\left| I_{i}(0) - y_{i}(0) \right| + \frac{C}{n^{1/4}} \right) \right) \geq 1 - 2e^{-c  \sqrt{n}},
\]
where $C, c > 0$ are constants depending only on $T$.
\end{proposition}

\begin{remark}\label{rem:explicit-solution}
Note that for $0\le y \le x$ we have $\lambda_x(t,y) \ge 0$, moreover $\lambda_x(t,x) = 0$. Together with Lemma \ref{le:Lipschitz-condition} this implies that for $0\le c_i \le \frac{i}{n}$ the ODE in the proposition above has a unique solution. We will need its explicit form for the initial condition $y_{i}(0) = x(1-a)$ (for $x,a \in [0,1]$), in which case we obtain
\begin{equation}\label{eq:explicit-solution}
y_{i}(t) = \frac{1}{t} \log \left( a + (1-a)e^{tx} \right).
\end{equation}
\end{remark}

\begin{proof}[Proof of Proposition \ref{prop:evolution-of-inversions}]
In this proof $C$ will denote a constant which depends only on $T$. The value of this constant may change between occurrences.

Fix $i \in [n]$. Since $p_{i}(\cdot,t)$ is the jump intensity of $\ell_{i}^{n}(t)$ at time $t$, we have for any $t \in [0,T]$
\begin{equation}\label{eq:l-counting-process}
\ell_{i}^{n}\left( e^{t/n} \right) - \ell_{i}^{n}(1) = M^{i}_{e^{t/n}} + \int\limits_{1}^{e^{t/n}} p_{i}\left( \ell_{i}^{n}(s),s \right) ds,
\end{equation}
where $\left(M^{i}_t\right)_{t \geq 1} = (\ell_i^n(t) - \ell_i^n(1) - \int_1^t p_i(\ell_i(s),s)ds)_{t\ge 1}$ is a martingale. Dividing by $n$ and performing the change of variables $s = e^{u/n}$ in the integral leads to
\[
I_{i}(t) - I_{i}(0) = \frac{1}{n}M^{i}_{e^{t/n}} + \int\limits_{0}^{t} \frac{1}{n^2} e^{u/n} p_{i}\left( \ell_{i}^{n}\left(e^{u/n}\right), e^{u/n} \right) du.
\]
Now we employ the approximation \eqref{eq:intensity-approximation} to get
\begin{multline*}
  \Big|\frac{1}{n^2} e^{u/n} p_{i}\left( \ell_{i}^{n}\left(e^{u/n}\right), e^{u/n} \right) - \lambda_{\frac{i}{n}}(I_i(u),u)\Big|\\
\le e^{u/n}\Big| \frac{1}{n^2} p_{i}\left( \ell_{i}^{n}\left(e^{u/n}\right), e^{u/n} \right) - \lambda_{\frac{i}{n}}(I_i(u),u)\Big| + 
\lambda_{\frac{i}{n}}(I_i(u),u) \Big|e^{u/n} - 1\Big| \le \frac{C}{n}.
\end{multline*}
As a consequence
\[
I_{i}(t) - I_{i}(0) = \frac{1}{n}M^{i}_{e^{t/n}} + \int\limits_{0}^{t} \lambda_{\frac{i}{n}}\left(I_i(u), u\right) du + \varepsilon_n(t),
\]
where $\varepsilon_n(t)$ is a stochastic process bounded in absolute value by $\frac{C}{n}$.

Now, since $y_i(t)$ satisfies the relevant ODE, we have
\[
y_i(t) - y_i(0) = \int\limits_{0}^{t} \lambda_{\frac{i}{n}}\left(y_i(s),s\right) ds,
\]
so
\[
I_i(t) - y_i(t) = I_i(0) - y_i(0) + \frac{1}{n}M^{i}_{e^{t/n}} + \int\limits_{0}^{t} \left( \lambda_{\frac{i}{n}}(I_i(s),s) - \lambda_{\frac{i}{n}}(y_i(s),s) \right) ds + \varepsilon_n(t).
\]
We get the following inequality
\[
\left|I_i(t) - y_i(t)\right| \leq |I_i(0) - y_i(0)| + \left| \frac{1}{n}M^{i}_{e^{t/n}} \right| + |\varepsilon_n(t)| + \int\limits_{0}^{t} \left| \lambda_{\frac{i}{n}}(I_i(s),s) - \lambda_{\frac{i}{n}}(y_i(s),s) \right| ds.
\]
Since $\lambda_{\frac{i}{n}}(\cdot,t)$ is $K$-Lipschitz on $[0,1]$ with $K$ depending only on $T$ (see Lemma \ref{le:Lipschitz-condition} from the Appendix), we can further write
\[
\left|I_i(t) - y_i(t)\right| \leq |I_i(0) - y_i(0)| + \left| \frac{1}{n}M^{i}_{e^{t/n}} \right| + |\varepsilon_n(t)| + K \int\limits_{0}^{t} \left| I_i(s) - y_i(s) \right| ds.
\]
Taking the supremum gives
\[
\sup\limits_{t \in [0,T]} |I_i(t) - y_i(t)| \leq |I_i(0) - y_i(0)| + \sup\limits_{t \in [0,T]} \left| \frac{1}{n}M^{i}_{e^{t/n}} \right| + \sup\limits_{t \in [0,T]} |\varepsilon_n(t)| + K \int\limits_{0}^{T} \sup\limits_{0\leq s \leq t} \left| I_{i}(t) - y_i(t) \right| dt
\]
and now an application of Gronwall's inequality leads to
\[
\sup\limits_{t \in [0,T]} |I_i(t) - y_i(t)| \leq \left( |I_i(0) - y_i(0)| + \sup\limits_{t \in [0,T]} \left| \frac{1}{n}M^{i}_{e^{t/n}} \right| + \sup\limits_{t \in [0,T]} |\varepsilon_n(t)| \right) e^{KT}.
\]
We have $\sup_{t\le T} |\varepsilon_n(t)| \le \frac{C}{n}$. Taking into account that the intensity $p_i(j,e^{t/n})$ is bounded on $[0,T]$ by $Cn^2$ we can now apply Lemma \ref{le:concentration-counting-process} from the Appendix to $N_t = \ell_i(1+t) - \ell_i(1)$, $a = Cn^2$, $s = e^{T/n}-1$, and $u = n^{3/4}$, and obtain (using the notation of the lemma)
\begin{displaymath}
  \Pp\Big(\sup_{t \in [0,T]} \frac{1}{n}|M^i_{e^{t/n}}| \ge \frac{1}{n^{1/4}}\Big) = \Pp(\sup_{t \in [0,s]} |Z_t| \ge n^{3/4}) \le 2\exp\Big(-\frac{n^{3/2}}{4C(e^{T/n}-1)n^2 + 2n^{3/4}/3}\Big).
\end{displaymath}
The right-hand side of the above inequality is clearly bounded by  $2\exp(- c\sqrt{n})$, where $c$ is a positive constant depending only on $T$. To finish the proof of the proposition it is now enough to take the union bound over $i \in [n]$ and adjust the constants.
\end{proof}

\end{subsection}

\begin{subsection}{Proof of Theorem \ref{th:main-theorem-global}}\label{subsec:global-proof}
With the previous results we can now prove the global limit for the Mallows process.

\begin{proof}[Proof of Theorem \ref{th:main-theorem-global}]

In the proof $C_\cdot,c_\cdot$ will denote constants depending only the parameters listed in the subscript. Their values may change between occurrences.

We will focus first on the behavior of $X_i^n(t)$ for $t \in [0,T]$. At the end of the proof we will use the symmetry of the Mallows distribution to pass to the interval $[-T,T]$.

Consider the partition $0 = t_0 < t_1 < \ldots < t_N = T$ of $[0,T]$ into intervals of length $t_{l+1} - t_{l} = \frac{T}{N}$ for $N = n^7$. By the remark following equation \eqref{eq:intensity-approximation} the intensity of each counting process $\ell^{n}_{i}(t)$ is bounded by $C_Tn^2$, which easily implies that with high probability the whole process $(I_1(t), \ldots, I_n(t))$ makes at most one jump in each interval $[t_l, t_{l+1})$. Indeed, the probability that $\ell_i^n$ makes a jump in an interval of length $\frac{1}{N}$ is at most $\frac{C_Tn^2}{N}$ and by the strong Markov property, the probability of two jumps is at most $\frac{C_T^2 n^4}{N^2}$. Taking the union bound over all $N$ intervals and $i=1,\ldots,n$ shows that the probability that the process makes more than one jump in any of the intervals is bounded by $\frac{C_T^2 n^4 Nn}{N^2} = \frac{C_T^2 }{ n^2} \to 0$.

Thus with high probability we have for each $i=1, \ldots, n$ and $l=0,\ldots,N-1$
\begin{align*}
& \sup\limits_{t \in [t_l, t_{l+1})} \left| X_{i}^{n}(t) - z_{\frac{i}{n}, X_{i}^{n}(0)}(t) \right| \leq
\max \left\{ \left|X_{i}^{n}(t_l) - z_{\frac{i}{n}, X_{i}^{n}(0)}(t_{l}) \right|, \left|X_{i}^{n}(t_{l+1}) - z_{\frac{i}{n}, X_{i}^{n}(0)}(t_{l+1}) \right| \right\} \\
& + \max\left\{ \sup\limits_{t \in [t_l, t_{l+1})} \left| z_{\frac{i}{n}, X_{i}^{n}(0)}(t_l) - z_{\frac{i}{n}, X_{i}^{n}(0)}(t)\right|, \sup\limits_{t \in [t_l, t_{l+1})} \left| z_{\frac{i}{n}, X_{i}^{n}(0)}(t_{l+1}) - z_{\frac{i}{n}, X_{i}^{n}(0)}(t)\right| \right\}.
\end{align*}
Since $z_{x,a}(t)$ is continuous for $t \in [0,T]$ (uniformly in $x$ and $a$, see Lemma \ref{le:equicontinuity-of-z}), the second term is bounded by a deterministic quantity converging to $0$ as $n \to \infty$, which altogether implies
\[
\sup\limits_{t \in [0,T]} \left| X_{i}^{n}(t) - z_{\frac{i}{n}, X_{i}^{n}(0)}(t) \right| \leq \max\limits_{l = 0, \ldots, N} \left|X_{i}^{n}(t_l) - z_{\frac{i}{n}, X_{i}^{n}(0)}(t_{l}) \right| + o(1).
\]
Thus it is enough to show that $\left|X_{i}^{n}(t) - z_{\frac{i}{n}, X_{i}^{n}(0)}(t) \right|$ is small with high enough probability (uniformly over $t$) for any given $i=1,\ldots,n$ and $t \in [0,T]$. 

From now on we fix $t \in [0,T]$ and write for brevity $x_i = \frac{i}{n}$, $a_i = X_{i}^{n}(0)$.

We first show that specifying $a_i$ determines the rescaled number of inversions of $i$, $I_{i}(t)$, up to a small error. Recall the quantity $\Delta_R$, defined in \eqref{eq:definition-of-Delta}. For fixed $\varepsilon > 0$ by Corollary \ref{cor:ldp} applied to $R_i = \left[0, x_i \right] \times \left[a_i, 1\right]$ and $\beta = 0$ we get (noting that $\rho_{0}(x,y) \equiv 1$)
\[
\left| \Delta_{R_i}(\sigma_{1}^{n}) - x_i(1-a_i) \right| < \varepsilon
\]
with probability at least $1 - 2 e^{-c_{T,\varepsilon}n}$. Since $\Delta_{R_i}(\sigma_{1}^{n}) = \frac{1}{n} \left| \left\{ j \in [n] \, \colon \, j \in [0, x_i n], \sigma_{1}^{n}(j) \in [a_i n, n]  \right\}  \right| =  \frac{1}{n}\left| \left\{ j \in [n] \, \colon \, j \leq i, \, \sigma_{1}^{n}(j) \geq \sigma_{1}^{n}(i)  \right\}  \right|$, this differs from $\frac{1}{n} \ell_{i}^{n}(1) = I_{i}(0)$ by $\frac{1}{n}$ and thus
\begin{equation}\label{eq:inversion-at-0}
\left| I_{i}(0) - x_i (1-a_i) \right| < \varepsilon + \frac{1}{n}.
\end{equation}
Next, by Proposition \ref{prop:evolution-of-inversions} applied with $c_i = x_i(1-a_i)$ we get that with probability at least $1 - 2e^{-c_{T}\sqrt{n}}$, 
\begin{equation}\label{eq:inversions-at-t}
\left| I_{i}(t) - y_{i}(t) \right| < C_T \left( \left| I_{i}(0) - y_{i}(0) \right| + \frac{C_T}{n^{1/4}} \right),
\end{equation}
where $y_{i}(t)$ is the solution of
\[
\begin{cases}
& \dot{y}_{i}(t) = \lambda_{x_i}(y_{i}(t),t), \\
& y_{i}(0) = x_i(1-a_i),
\end{cases}
\]
given by the explicit formula \eqref{eq:explicit-solution} with $a = a_i$. Since we can choose $\varepsilon$ to be arbitrarily small, by combining \eqref{eq:inversion-at-0} with  \eqref{eq:inversions-at-t} together with a union bound we obtain that for any $\varepsilon > 0$ and large $n$, with probability at least $1 - 2e^{-c_{T,\varepsilon}\sqrt{n}}$ we have
\begin{equation}\label{eq:inversions-at-t-final}
\left| I_{i}(t) - y_{i}(t) \right| < \varepsilon
\end{equation}
for all $i \le n$.

Now that we know $I_{i}(t)$, we show that $X_{i}^{n}(t)$ itself is essentially determined by $x_i$ and $a_i$.

Let us write for simplicity $\sigma = \sigma^{n}_{e^{t/n}}$. By Lemma \ref{lm:density-lower-bound} we have that for $t \in [0,T]$ the density $\rho_{t}(x',y')$ is bounded from below by some $A_T > 0$ depending only on $T$.

By Corollary \ref{cor:ldp} applied with $\delta = \varepsilon$, $\beta = t$ we get that
\begin{equation}\label{eq:ldp-slices}
 \max\limits_{u,v\in[0,1]} \bigg| \Delta_{[0,u] \times [v,1]}(\sigma) - \int\limits_{v}^{1} \int\limits_{0}^{u} \rho_{t}(x',y') \, dx' \, dy' \bigg| < \varepsilon
\end{equation}
with probability at least $1 - 2e^{-c_{T, \varepsilon}n}$. 

Since $I_i(t) = \Delta_{\left[0,x_i\right]\times [X_i^n(t),1]} - \frac{1}{n}$, we thus obtain that with the same high probability, for all $i \le n$,
\[
\bigg| I_i(t) - \int\limits_{X_i^n(t)}^{1} \int\limits_{0}^{x_i} \rho_{t}(x',y') \, dx' \, dy' \bigg| < \varepsilon + \frac{1}{n}.
\]
Combined with \eqref{eq:inversions-at-t-final}, this shows that for large $n$, with probability at least $1 - 2e^{-c_{T,\varepsilon}\sqrt{n}}$ for all $i \le n$,
\[
  \bigg|y_i(t) - \int\limits_{X_i^n(t)}^{1} \int\limits_{0}^{x_i} \rho_{t}(x',y') \, dx' \, dy' \bigg| < 3\varepsilon.
\]

For $x \in [0,1]$ consider now a function $F_x\colon [0,1] \to [0,x]$, given by
\begin{equation}\label{eq:function-f-rho}
F_x(z) = \int\limits_{z}^{1} \int\limits_{0}^{x} \rho_{t}(x',y') \, dx' \, dy'.
\end{equation}
The last inequality can be thus written as $|y_i(t) - F_{x_i}(X_i^n(t))| < 3\varepsilon$.

Recall that $\rho_t$ is bounded from below by $A_T>0$. Thus, for $x > 0$, $F_x$ is strictly decreasing and onto $[0,x]$. Moreover, for all $z_1,z_2 \in [0,1]$, $|F_x(z_2) - F_x(z_1)| \ge A_Tx|z_1-z_2|$. In particular, with high probability for all $i \le n$,
\[
  \big|F_{x_i}^{-1}(y_i(t)) - X_{i}^{n}(t)\big| \le \frac{n}{i} \cdot\frac{1}{A_T}\big|y_i(t) - F_{x_i}(X_{i}^{n}(t))\big| \le \frac{3n}{iA_T}\varepsilon.
\]
Thus for any $\alpha \in (0,1)$ and $\varepsilon > 0$,
\[
\Pp\left( \max\limits_{i \in (\alpha n, (1-\alpha)n)} \sup\limits_{t \in [0,T]} \left| X_{i}^{n}(t) - F_{x_i}^{-1}(y_i(t)) \right| > \varepsilon \right) \xrightarrow{n \to \infty} 0.
\]
To conclude the proof on the interval $[0,T]$ it now suffices to show that $F_{x_i}^{-1}(y_i(t)) = z_{x_i, a}(t)$, i.e., $y_{i}(t) = F_{x_i}(z_{x_i, a}(t))$. Since $\rho_t$ appearing in the formula \eqref{eq:function-f-rho} is explicitly known (recall formula \eqref{eq:definition-of-rho-beta}, one can compute $F_x(z)$ (e.g., using symbolic computation software such as Mathematica), which yields
\[
F_x(z) = \frac{1}{t} \log \frac{e^{xt}(e^t - 1)}{e^{(x+z)t} - e^{xt} - e^{zt} + e^t}.
\]
Since $y_i(t)$ is given by \eqref{eq:explicit-solution} and $z_{x_i,a}(t)$ by \eqref{eq:formula-for-global-limit}, one can perform the substitution and check the desired equality.

It remains to pass from $[0,T]$ to $[-T,T]$. Let $\rev_n$ be the order reversing permutation of $[n]$, $\rev_n(i) = n - i + 1$ for $i \in [n]$. For $t \ge 0$ define $\widetilde{\sigma}^n_t = \rev_n \circ \sigma^{n}_{1/t}$ (with the convention $\sigma^{n}_{\infty} = \rev_n$). Then, as one can easily check, we have $\widetilde{\sigma}^{n}_{t} \sim \pi_{n,t}$. Moreover, since time reversal preserves the Markov property, $\widetilde{\sigma}^n$ is also a Markov process. Clearly $\ell_i(\widetilde{\sigma}^n_t) = i - 1 - \ell_i(\sigma^{n}_{1/t})$, which implies that $\ell_{i}(\widetilde{\sigma}^n_t) \leq \ell_i(\widetilde{\sigma}^n_{t'})$ whenever $t \leq t'$ and that the processes $(\ell_i(\widetilde{\sigma}^n_t))_{t \in [0, \infty)}$ are independent over $i \in [n]$.

Consider now the process $\widehat{\sigma}^{n}_{t} = \widetilde{\sigma}^{n}_{t+}$. Since the probability that $\sigma^n$ makes a jump at a deterministic time $t$ is zero, $\widehat{\sigma}$ is a modification of $\widetilde{\sigma}$ and so it has the same finite-dimensional distributions. It is thus a c\`{a}dl\`{a}g Markov process with values in $\Ss_n$ whose marginal distribution at time $t$ is equal to $\pi_{n,t}$. It also holds that $\Inv(\widehat{\sigma}_{t}^{n}) \leq \Inv(\widehat{\sigma}_{t-}^{n}) + 1$, which together with independence and monotonicity properties of the inversion counts $\ell_{i}$ implies that it is a regular Markov Mallows process. By Corsini's uniqueness result mentioned in the introduction (\cite[Theorem 1.1]{corsini}) it has the same distribution of paths as $\sigma^n$. In particular, from what we have already proved on $[0,T]$ we get
\[
\Pp\left( \max\limits_{i \in (\alpha n, (1-\alpha)n)} \sup\limits_{t \in [0,T]} \bigg| \frac{1}{n} \widehat{\sigma}^{n}_{e^{t/n}}(i) - z_{\frac{i}{n}, \frac{1}{n} \widehat{\sigma}^{n}_{1}(i)} (t)  \bigg| > \varepsilon \right) \xrightarrow{n \to \infty} 0.
\]
Since the limit $z_{x,a}(t)$ is continuous in $t$ and $\widehat{\sigma}^n_{1}(i) = \widetilde{\sigma}^n_{1}(i)$ almost surely, the same holds for $\frac{1}{n}\widetilde{\sigma}^n_{e^{t/n}}(i) = 1 - X_{i}^{n}(-t)$. In other words, we have
\[
\Pp\left( \max\limits_{i \in (\alpha n, (1-\alpha)n)} \sup\limits_{t \in [-T,0]} \left| X_{i}^{n}(t) - \left( 1 - z_{\frac{i}{n},1 - X_{i}^{n}(0)}(-t) \right) \right| > \varepsilon \right) \xrightarrow{n \to \infty} 0.
\]
This ends the proof, since as one can easily check from \eqref{eq:formula-for-global-limit}
\[
  1 - z_{x,1-a}(-t) = z_{x,a}(t).
\]
\end{proof}
\end{subsection}
\end{section}

\begin{section}{Local limit}\label{sec:local-limit}

\begin{subsection}{The Mallows distribution on $\Z$}\label{subsec:mallows-on-z}
\paragraph{Mallows distribution on $\Z$ -- definition and existence.} To set the stage for the proof of the local limit of the Mallows process, we begin by recalling the construction of the \emph{Mallows distribution on $\Z$} due to Gnedin and Olshanski (\cite{gnedin-olshanskii}), which gives the analog of the Mallows distribution $\pi_{n,q}$ for permutations of $\Z$. We will only mention results which are directly relevant to the proof of Theorem \ref{th:main-theorem-local} -- the reader is referred to \cite{gnedin-olshanskii} for a more in-depth discussion.

Recall that $\GSs$ denotes the set of all permutations of $\Z$ and $\GSs^{bal} \subseteq \GSs$ is the set of all balanced permutations. Let $\sigma_{i,i+1} \in \GSs$ denote the transposition swapping adjacent indices $i, i+1 \in \Z$. Given $q > 0$, we say that a measure $\mu$ on $\GSs$ is \emph{right $q$-exchangeable} if for every $i \in \Z$ the pushforward $\mu_{i,i+1}$ of $\mu$ under the map $\sigma \mapsto \sigma \sigma_{i, i+1}$ is absolutely continuous with respect to $\mu$ and the Radon-Nikodym derivative $\frac{d\mu_{i,i+1}}{d\mu}(\sigma)$ equals $q^{\sgn\left( \sigma(i+1) - \sigma(i) \right)}$. Likewise, we say that $\mu$ is \emph{left $q$-exchangeable} if the analogous conditions hold under maps $\sigma \mapsto \sigma_{i, i+1} \sigma$ with the Radon-Nikodym derivative $q^{\sgn\left( \sigma^{-1}(i+1) - \sigma^{-1}(i) \right)}$.

One readily checks that Mallows distributions $\pi_{n,q}$ on $\Ss_n$ satisfy the appropriately adjusted right and left $q$-exchangeability properties. The existence of a probability measure with such properties on $\GSs$ is nontrivial and was proved in \cite{gnedin-olshanskii}. More precisely, we have the following result (\cite[Theorem 3.3.]{gnedin-olshanskii})
\begin{theorem}\label{thm:GO-Mallows-existence}
  Let $q \in (0,1)$. The two notions of $q$-exchangeability on $\GSs^{bal}$ coincide. There exists a unique probability measure $\Pi_q$ on $\GSs^{bal}$ which is both right and left $q$-exchangeable.
\end{theorem}

Furthermore, the Mallows distribution $\pi_{n,q}$ can be obtained by considering a restriction of $\Pi_{q}$ to the set of permutations of any interval $I \subseteq \Z$ of length $n$, where for $\sigma \in \GSs$ we identify $\sigma|_I$ with an element of $\Ss_n$ by relabeling the elements of $I$ and the elements of $\sigma(I)$ by the set $\{1,\ldots,n\}$ in increasing order.

Th measure $\Pi_q$ is called the Mallows measure on $\Z$ with parameter $q$. Let us remark that a related notion of the Mallows measure on the set of permutations of $\N$ was considered by Gnedin and Olshanski in an earlier work \cite{gnedin-olshanskii-1} (see also Remark \ref{rm:construction-on-n}).

\paragraph{Description of $\Pi_q$ by inversion counts.} For our purposes it will be important that the measure $\Pi_{q}$ can be described in terms of a family of independent geometric random variables. Recall that to sample a permutation $\sigma$ from the Mallows distribution $\pi_{n,q}$, one can first sample an inversion vector $(\ell_1, \ldots, \ell_n)$, with all coordinates independent and $\ell_i$ having a geometric distribution with parameter $1-q$, truncated at $i-1$, and then use the bijection between admissible inversion vectors and permutations of $[n]$. A somewhat analogous approach can be employed to sample permutations from $\Pi_q$. However, since on $\Z$ there is no smallest or largest element to start with, one needs a recursive procedure which we now describe.

For a permutation $\sigma \in \GSs$ we define the number of \emph{left inversions} of element $i$ as
\[
\ell_i(\sigma) = \left|\left\{ j \in \Z \, : \, j < i, \sigma(j) > \sigma(i) \right\}\right|.
\]
Likewise, the number of \emph{right inversions} of $i$ is defined as
\[
r_i(\sigma) = \left|\left\{ j \in \Z \, : \, j > i, \sigma(j) < \sigma(i) \right\}\right|.
\]
Note that a priori these numbers can be infinite. It is proved in \cite{gnedin-olshanskii} (Proposition 4.1) that $\ell_i(\sigma), r_i(\sigma)$ are in fact finite for all $i$ if the permutation $\sigma$ is balanced.

If one knows the left and right inversion numbers for a balanced permutation $\sigma \in \GSs^{bal}$, the permutation itself can be easily reconstructed by the following formula (\cite[Lemma 4.6]{gnedin-olshanskii})
\begin{equation}\label{eq:sigma-r-l}
\sigma(i) = i + r_i(\sigma) - \ell_i(\sigma).
\end{equation}
For $\sigma \sim \Pi_{q}$ the left inversion numbers are independent and identically distributed (\cite{gnedin-olshanskii}), with $\ell_i(\sigma)$ having geometric distribution with parameter $1-q$
\[
\Pp\left( \ell_i(\sigma) = j \right) = (1-q)q^{j}, \, \, \, j= 0, 1, 2,\ldots
\]
Because the roles of the left and right inversion numbers are symmetric (thanks to Theorem \ref{thm:GO-Mallows-existence}), the same applies to the right inversion numbers $r_i(\sigma)$. If one could compute the numbers $\ell_i(\sigma)$ from the (bi-infinite) sequence $\left(r_k(\sigma)\right)_{k \in \Z}$ or vice versa, a simple method for sampling from $\Pi_q$ would be available -- sample an infinite i.i.d. sequence of geometric variables, say $\left(\ell_k\right)_{k \in \Z}$, compute the corresponding $r_i$ and then employ formula \eqref{eq:sigma-r-l}.

Lemma 4.7 from \cite{gnedin-olshanskii} shows how to compute any $\ell_i$ recursively given $(r_k)_{k \in \Z}$. However, to be consistent with the description of the finite Mallows distributions $\pi_{n,q}$ in terms of the left rather than the right inversion vector, as we do in this paper and as is also done in \cite{corsini}, we would prefer to treat $\ell_i$ as independent and compute $r_i$ given $\left(\ell_k\right)_{k \in \Z}$. After making the necessary adjustments, the lemma reads as follows
\begin{lemma}[\cite{gnedin-olshanskii}, Lemma 4.7, version for left inversions]\label{lm:recursive-formula-r}
Let $\sigma \in \GSs^{bal}$ be a balanced permutation with the associated left and right inversion numbers $\ell_i = \ell_i(\sigma)$, $r_i = r_i(\sigma)$. Given $i \in \Z$, let $\ell_{i}^{(i)} = \ell_{i}$ and for $j \geq i$ define recursively
\[
\ell_{i}^{(j+1)} = \ell_{i}^{(j)} + \id_{\{ \ell_{i}^{(j)} \geq \ell_{j+1} \}}.
\]
Then we have
\begin{equation}\label{eq:recursive-formula-r}
r_i = \sum\limits_{j : \, j > i} \id_{\{ \ell_{i}^{(j)} < \ell_j \}} = \sum\limits_{j : \, j > i} \id_{\{ \ell_{i}^{(j-1)} = \ell_{i}^{(j)} \}}.
\end{equation}
\end{lemma}

Let us remark that $\ell_i^{(j)}$ can be characterized by the property that there are exactly $\ell_i^{(j)}$ elements greater than $\sigma(i)$ in the sequence $(\sigma(k))_{k\le j}$. In other words , $\ell_i^{(j)} + 1$ is the reverse rank of $\sigma(i)$ in $(\sigma(k))_{k\le j}$. %
The last fact that we will need is a characterization of admissible inversion numbers. Again we state the lemma with the roles of left and right inversions switched with respect to the statement in \cite{gnedin-olshanskii}.

\begin{lemma}[\cite{gnedin-olshanskii}, Theorem 7.1, version for left inversions]\label{lm:inversions-characterization}
A nonnegative integer sequence $\left( \ell_i \right)_{i \in \Z}$ occurs as a sequence of left inversions counts for some permutation $\sigma \in \GSs^{bal}$ if and only if the following two conditions hold:
\begin{enumerate}[(i)]
\item the values $\left( r_i \right)_{i \in \Z}$ determined from \eqref{eq:recursive-formula-r} are finite,
\item $\ell_i = 0$ for infinitely many $i \leq 0$.
\end{enumerate}
Under these conditions such $\sigma$ is unique and has right inversion counts $\left( r_i \right)_{i \in \Z}$.
\end{lemma}

\begin{theorem}[\cite{gnedin-olshanskii}, Theorem 4.3, version for left inversions] \label{thm:GO-Mallows-construction}
Let $\ell_i$, $i \in \Z$, be independent random variables distributed according to the geometric distribution with parameter $1-q$. Let $r_i$ be defined by \eqref{eq:recursive-formula-r}. Then with probability one, all $r_i$'s are finite and formula \eqref{eq:sigma-r-l} defines a balanced permutation $\sigma$. Moreover, $\sigma$ has law $\Pi_q$.
\end{theorem}

\end{subsection}

\begin{subsection}{Construction of the limiting process}\label{subsec:construction-of-limit}
Based on the description of the Mallows distribution in terms of the left and right inversions numbers we can now construct the limiting process $\left( \Sigma_{t} \right)_{t \in [0,1)}$. Similarly as for the Mallows process on $\Ss_n$, the construction will be based on a family of independent processes defining the left inversion numbers. It will be slightly more convenient to first define these processes for $t \in [0, \infty)$ and then perform a time change to $[0,1)$.

Let us thus consider a family of independent  c\`{a}dl\`{a}g time-homogeneous Markov birth processes $( \widetilde{\ell}_i(t) )_{t \in [0,\infty)}$, $i \in \Z$, with $\widetilde{\ell}_i(t)$ taking values in nonnegative integers, and having rate of jumping from $j$ to $j+1$ at time $t$ given by $\widetilde{q}(j) = j+1$. We take $\widetilde{\ell}_i(0) = 0$. Let $\tau : [0, 1) \to [0,\infty)$ be a (deterministic) change of time given by $\tau(t) = - \log(1-t)$ and let us define for any $t \in [0,1)$, $\ell_i(t) = \widetilde{\ell}_i(\tau(t))$. It is readily seen that $( \ell_i(t) )_{t \in [0,1)}$ is a time-inhomogeneous Markov birth process with birth rate $q(j, t) = \widetilde{q}(j)\tau'(t)$ equal to
\begin{equation}\label{eq:limiting-intensity}
q(j,t) = \frac{j+1}{1-t}, \, \, \, j = 0, 1, 2, \ldots
\end{equation}
The reader may notice that the jump rate is the same as for the inversions of Mallows process on $\Ss_n$, up to a term which is small if $t < 1$ and $n \to \infty$ (recall formula \eqref{eq:def-of-rates}). This observation will form the basis of the proof of Theorem \ref{th:main-theorem-local}.

\begin{proposition}
The birth process with birth rates $q(j,t)$ given by \eqref{eq:limiting-intensity} and starting from $0$ is non-explosive for $t \in [0,1)$. Moreover, the marginal distribution of the process at time $t \in [0,1)$ is the geometric distribution with parameter $1-t$.
\end{proposition}

\begin{proof}
We need to show that the process almost surely makes only finitely many jumps on any compact interval $[0,T]$. It is enough to show that the same holds for the process before the time change, that is, the process on $[0,\infty)$ with jump rates $\widetilde{q}(j) = j+1$. This follows directly from the non-explosion criterion for time-homogeneous birth processes (\cite[Theorem 6.8.17]{grimmett-stirzaker}), since $\sum\limits_{j=0}^{\infty}\widetilde{q}(j)^{-1} = \sum\limits_{j=0}^{\infty}\frac{1}{j+1} = + \infty$.

To find the marginal distribution at time $t$ we solve the Chapman--Kolmogorov forward equation -- if $P_{0,j}(t)$ is the probability of the process being at $j$ at time $t$, starting from $0$, then we have for any $t \in [0,1)$
\[
P'_{0,0}(t) = - q(0,t)P_{0,0}(t)
\]
and
\[
P'_{0,j}(t) = q(j-1,t)P_{0,j-1}(t) - q(j,t)P_{0,j}(t)
\]
for $j \geq 1$, with $P_{0,j}(0) = \delta_{0j}$. One then directly checks that $P_{0,j}(t) = (1-t)t^{j}$, $j \geq 0$, is a solution, and since it is unique, the marginal distribution at time $t$ is geometric with parameter $1-t$.
\end{proof}

\begin{remark}
From the above proposition it follows in particular that each of the processes $\ell_i$ makes infinitely many jumps in the interval $[0,1)$.
\end{remark}

Now for a given $t \in [0,1)$ we define $r_{i}(t)$ to be the quantity $r_i$ computed from \eqref{eq:recursive-formula-r} when the sequence $\left( \ell_{i}(t) \right)_{i \in \Z}$ is used as $\left(\ell_{i}\right)_{i \in \Z}$, i.e., for $j \geq i$ we take
\begin{equation}\label{eq:recursive-formula-ljt}
\ell_{i}^{(j+1)}(t) = \ell_{i}^{(j)}(t) + \id_{\{ \ell_{i}^{(j)}(t) \geq \ell_{j+1}(t) \}},
\end{equation}
with $\ell_{i}^{(i)}(t) = \ell_{i}(t)$, and define
\begin{equation}\label{eq:recursive-formula-rt}
r_i(t) = \sum\limits_{j : \, j > i} \id_{\{ \ell_{i}^{(j)}(t) < \ell_j(t) \}} = \sum\limits_{j : \, j > i} \id_{\{ \ell_{i}^{(j-1)}(t) = \ell_{i}^{(j)}(t) \}}.
\end{equation}

Finally, for any $t \in [0,1)$ and $i \in \Z$ we let
\begin{equation}\label{eq:formula-for-sigma}
\Sigma_t(i) = i + r_i(t) - \ell_i(t).
\end{equation}

\begin{proposition}\label{prop:construction-of-limiting-sigma}
The process $\left( \Sigma_t \right)_{t \in [0,1)}$ defined by formula \eqref{eq:formula-for-sigma} is a c\`{a}dl\`{a}g $\GSs^{bal}$-valued  process whose marginal distribution at time $t$ is the Mallows distribution $\Pi_{t}$ with parameter $1-t$.
\end{proposition}

\begin{proof}
In the proof we will assume without loss of generality that on the whole probability space the trajectories of $\ell_i$ are c\`{a}dl\`{a}g, nondecreasing on $[0,1)$ and have jumps equal to one.

For any $t \in (0,1)$, $\ell_i(t)$, $i \in \Z$, are i.i.d. geometric variables with parameter $1-t$. Thus, by Theorem \ref{thm:GO-Mallows-construction}, $\Sigma_t$ is with probability one a well-defined balanced permutation, moreover it is distributed according to $\Pi_t$.

To conclude the proof, it is enough to show that for any $T \in (0,1)$, there exists an event $A_T$ with $\Pp(A_T) = 1$, such that on $A_T$ the trajectories $t \mapsto \Sigma_t$, $t \in [0,T]$, are $\GSs^{bal}$-valued and c\`{a}dl\`{a}g.

Let
\begin{align*}
A_T = \Big\{&\ell_i(T) = 0 \textrm{ for infinitely many $i \ge 0$ and for infinitely many $i \le 0$},\\
&\; r_i(T) < \infty \textrm{ for all $i \in \Z$} \Big\}.
\end{align*}

Theorem \ref{thm:GO-Mallows-construction} implies that $\Pp(A_T)=1$. From now on we will be working with a fixed $\omega \in A_T$, which most of the time will be suppressed in the notation. We will first prove that for any $i \in \Z$, $t\mapsto \Sigma_t(i)$ is a c\`{a}dl\`{a}g function for $t \in [0,T]$.

Fix thus $i \in \Z$ and let $i_1 \ge i$ be such that $\ell_{i_1}(T) = 0$. Recall \eqref{eq:recursive-formula-ljt} and \eqref{eq:recursive-formula-rt}.
Clearly, for any $t \le T$ we have $\ell_i^{(i_1)}(t) \ge 0 = \ell_{i_1}(T) = \ell_{i_1}^{(i_1)}(T)$. It follows by induction that for any $j \ge i_1$,
\begin{align}\label{eq:elle-comparison}
  \ell_i^{(j)}(t) \ge \ell_{i_1}^{(j)}(T).
\end{align}
Indeed, if \eqref{eq:elle-comparison} holds for some $j$, then
\begin{displaymath}
  \ell_i^{(j+1)}(t) = \ell_i^{(j)}(t) + \id_{\{ \ell_{i}^{(j)}(t) \geq \ell_{j+1}(t) \}} \ge \ell_{i_1}^{(j)}(T) + \id_{\{ \ell_{i_1}^{(j)}(T) \geq \ell_{j+1}(T) \}} = \ell_{i_1}^{(j+1)}(T),
\end{displaymath}
where the inequality follows from the induction assumption and the estimate $\ell_{j+1}(t) \le \ell_{j+1}(T)$.

From \eqref{eq:elle-comparison} it follows in particular that for $j \ge i_1$, if $\ell_{i_1}^{(j)}(T) \ge \ell_{j}(T)$, then $\ell_i^{(j)}(t) \ge \ell_j(t)$. By the finiteness of $r_{i_1}(T)$, there exists $i_2 \ge i_1$ such that the former inequality holds for all $j > i_2$. Thus for all $t \le T$ and $j > i_2$ we have
\begin{displaymath}
  \ell_i^{(j)}(t) \ge \ell_j(t).
\end{displaymath}
In particular, it follows from \eqref{eq:recursive-formula-ljt} and \eqref{eq:recursive-formula-rt} that for $t \le T$, $r_i(t)$ is fully determined by $\ell_i(t),\ldots,\ell_{i_2 + 1}(t)$, i.e., there exists a function $f = f_{T,i, \omega} \colon \N^{i_2-i+2} \to \N$ such that for all $t \leq T$, $r_i(t) = f(\ell_i(t),\ldots,\ell_{i_2 + 1}(t))$. Since $\ell_i(\cdot)$ are c\`{a}dl\`{a}g and take integer values, $r_i(\cdot)$ is also c\`{a}dl\`{a}g on $[0,T]$ and by \eqref{eq:formula-for-sigma}, so is the trajectory $t\mapsto \Sigma_t(i)$.

Moreover, for all $t \le T$, $\ell_i(t) = 0$ for infinitely many $i \le 0$ and $r_i(t) < \infty$ for all $i\in \Z$, which by Lemma \ref{lm:inversions-characterization} implies that $\Sigma_t \in \GSs^{bal}$. The same holds for $\ell_i(t-)$ and $r_i(t-)$, which proves that for all $t \in (0,T]$ the function $\Sigma^{-}_{t} \colon \Z \to \Z$ defined by $\Sigma^{-}_{t}(i) = i + r_i(t-) - \ell_i(t-)$ is also an element of $\GSs^{bal}$. Since the topology on $\GSs^{bal}$ that we consider is that of pointwise convergence, the c\`{a}dl\`{a}g property of $\ell_i(\cdot)$ and $r_i(\cdot)$ implies that $t \mapsto \Sigma_t$ is also a c\`{a}dl\`{a}g $\GSs^{bal}$-valued function. This ends the proof of the proposition.
\end{proof}

\begin{remark}
Thanks to the bijection between admissible inversion numbers and balanced permutations (Lemma \ref{lm:inversions-characterization}), the process $\left( \Sigma_t \right)_{t \in [0,1)}$ is a Markov process, since $\ell_{i}(t)$, $i \in \Z$, are Markov processes.
\end{remark}

In the proposition below we show that $\Sigma_t$ is a transposition process, that is, whenever it jumps, the permutations before and after the jump differ by a transposition $\sigma_{ij}$ of two (not necessarily adjacent) elements. An analogous property for the Mallows process on $\Ss_n$ was proved in \cite[Theorem 1.2]{corsini}.

\begin{proposition}\label{prop:sigma-is-transposition}
With probability one the following holds -- if the process $\left( \Sigma_t \right)_{t \in [0,1)}$ jumps at time $s \in (0,1)$, then letting $\Sigma_{s-} = \sigma$, $\Sigma_{s} = \sigma'$ we have $\sigma^{-1} \circ \sigma' = \sigma_{ij}$ for some $i,j \in \Z$, where $\sigma_{ij}$ is the transposition swapping $i$ and $j$.
\end{proposition}

\begin{proof}
Since the process $\Sigma_t$ jumps only when one of the processes $\ell_j(t)$ jumps and almost surely the jumps are not simultaneous, there exists exactly one $i \in \Z$ such that $\ell_{i}(\sigma') = \ell_{i}(s) = \ell_{i}(s-) + 1 = \ell_{i}(\sigma) + 1$ and $\ell_{j}(\sigma') = \ell_{j}(\sigma)$ for $j \neq i$. Note that by construction of the process $\Sigma_t$ all the left and right inversion numbers $\ell_{j}(\sigma)$, $\ell_{j}(\sigma')$, $r_{j}(\sigma)$, $r_{j}(\sigma')$ are almost surely finite.

Let us suppose that $\ell_{i}(\sigma) = k$ for some $k \geq 0$, so that $\sigma(i)$ is the $(k+1)$-st largest element in the sequence $(\sigma(j))_{j \leq i}$. Let $\sigma(i')$ be the $(k+2)$-nd largest element in this sequence (note that such an element must exist since otherwise $\ell_i(\sigma)$ would not be finite). Let $\sigma_{i i'}$ be the transposition swapping $i$ and $i'$ and consider the permutation $\widetilde{\sigma}= \sigma \circ \sigma_{i i'}$. Note that $\widetilde{\sigma}$ has exactly the same left inversion numbers $\ell_{j}$ as $\sigma$ except that $\ell_{i}(\widetilde{\sigma}) = \ell_{i}(\sigma) + 1$. It is also easy to see that the right inversion numbers $r_{j}(\widetilde{\sigma})$ are finite, since $r_{j}(\sigma)$ were finite and this property is preserved after transposing two elements. As by Lemma \ref{lm:inversions-characterization} the left inversion numbers uniquely characterize balanced permutations and $\sigma'$, $\widetilde{\sigma}$ have the same left inversion sequences, we must have $\sigma' = \widetilde{\sigma}$, i.e., $\sigma^{-1} \circ \sigma' = \sigma_{ii'}$, which ends the proof.
\end{proof}
\end{subsection}

\begin{remark}\label{rm:construction-on-n}
The process $\left( \Sigma_t \right)_{t \in [0,1)}$ constructed in this subsection will serve as the local limit of the Mallows process for elements which start far from the boundary of the interval $\{1, \ldots, n\}$ (recall the condition $k_n$, $n - k_n \to \infty$ in the statement of Theorem \ref{th:main-theorem-local}). For elements close to $1$ or $n$ one should instead consider an analogous process, but having values in permutations of $\N$ or $-\N$ rather than $\Z$. For the construction, say on $\N$, it is enough to replace $\ell_{i}(t)$, $i \in \Z$, with processes $\ell_i(t)$, $i \in \N$, such that for fixed $t$ the marginal distribution of $\ell_i(t)$ is the geometric distribution with parameter $1-t$ truncated at $i-1$ (as in the Mallows distribution on $\Ss_n$).  One then easily checks that the permutation $\Sigma_t$ constructed by means of formula \eqref{eq:formula-for-sigma} indeed has Mallows distribution on $\N$ with parameter $t$ (see \cite{gnedin-olshanskii-1} for the construction of this measure). The case of $-\N$ is analogous.
\end{remark}

\begin{subsection}{Proof of Theorem \ref{th:main-theorem-local}}\label{sec:proof-of-local-limit}

We can now proceed to the proof of the local limit. Recall that $\left( \sigma^{n}_t \right)_{t \in [0,1)}$ is the Mallows process on $\Ss_n$, with $\ell_{i}^{n}(t) = \ell_{i}\left(\sigma_{t}^{n}\right)$ denoting the $i$-th coordinate of the left inversion vector at time $t$, $i=1, \ldots, n$. Recall the definition \eqref{eq:shifted-process-n} of the shifted process $\Sigma^{n}_t$ -- for a sequence $k_n \in \Z$ such that $k_n, n - k_n \to \infty$, and any $i \in \Z$, $t \in [0,1)$ we let
\[
\Sigma^{n}_{t}(i) = \sigma^{n}_{t}\left( k_n + i \right) - k_n,
\]
where we take $\Sigma^{n}_{t}(i) = i$ if $k_n + i \notin [n]$.
\begin{proof}[Proof of Theorem \ref{th:main-theorem-local}]

We will construct a coupling $\widetilde{\Sigma}, \widetilde{\Sigma}^n$, $n \ge 1$, of the processes $\Sigma$ and $\Sigma^n$ such that
with probability one the following holds -- for any $T \in [0,1)$ and $i \in \Z$ there exists $n_0$ such that for all $n > n_0$ and $t \in [0,T]$ we have $\widetilde{\Sigma}^n_t(i) = \widetilde{\Sigma}(i)$. This in particular implies that for any $T \in [0,1)$, $\sup_{t \in [0,T]} \rho(\widetilde{\Sigma}^n(t),\widetilde{\Sigma}(t)) \to 0$ with probability one, where $\rho$ metrizes pointwise convergence of permutations. It follows that $\widetilde{\Sigma}^n \to \widetilde{\Sigma}$ a.s. in $\DD$ and as a consequence $\Sigma^n$ converges weakly to $\Sigma$.

Note that the existence of such a coupling also implies that for any $T \in [0,1)$ and $N \in \N$, the processes $(\Sigma^n_t(i))_{i \in \{-N,\ldots,N\},t \in [0,T]}$ converge in the total variation distance to the process $(\Sigma_t(i))_{i \in \{-N,\ldots,N\}, t\in [0,T]}$, as stated in Remark \ref{re:total-variation}.

In what follows to simplify the notation we will denote the constructed coupled processes just by $\Sigma^n$ and $\Sigma$.

Let thus $\Sigma$ be a Mallows process on $\GSs^{bal}$ constructed via \eqref{eq:formula-for-sigma}, where $\ell_i$ have jump rates $q(\cdot,\cdot)$ given by \eqref{eq:limiting-intensity} (in particular $\ell_i(t)$ has the geometric distribution with parameter $1-t$). Recall also the definition \eqref{eq:def-of-rates} of jump rates $p_i(j,t)$ for the inversion counts of the Mallows process $\sigma^n$. Additionally, let us set $p_i(j,t) = 0$ for $j \ge i$.

The idea of the proof is to define $\sigma^n$ using the inversion counts $\ell^n_i$ constructed by a dependent thinning of the jumps of $\ell_{i - k_n}$. Each jump of the latter process will be preserved with certain (conditional) probability, in such a way that the birth process defined by the preserved jumps will be a Markov process with jump rates $p_i$. The thinning will be implemented through independent random variables, uniform on $(0,1)$. Since the jump rates $p_{i}$ are close to $q$ for large $n$, this will allow us to show that with probability one the process $\ell^n_i$ constructed this way coincide with $\ell_{i-k_n}$ for any interval $[0,T]\subseteq [0,1)$ and $n$ large enough (depending on $T$).

Let thus $U_{i,j}$, $i\in \Z, j\in \Z_+$ be a family of i.i.d. random variables uniform on $(0,1)$, independent of $\Sigma$. For $i \in \Z$ and $j\in \Z_+$ let $T_{i,j}$ be the time of the $j$-th jump of $\ell_i$ on $[0,1)$ and let $T_{i,0} = 0$.

Let us also for $i, n\in \Z_+, j \in \N$ define inductively the random variables $X^n_{i,j}$ as
\begin{align*}
  X^n_{i,0} &= 0
\end{align*}
and
\begin{align}\label{eq:definition-of-X}
  X^n_{i,j+1}  =
  \begin{cases}
  X^n_{i,j} + 1, &\textrm{ if } U_{i-k_n,j+1} \le \frac{p_i(X^n_{i,j},T_{i-k_n,j+1})}{q(\ell_{i-k_n}(T_{i-k_n,j+1}-),T_{i-k_n,j+1})},\\
  X^n_{i,j}, &\textrm{ otherwise.}
  \end{cases}
\end{align}
The variables $X^n_{i,j}$ will keep track of which jumps of the process $\ell_{i - k_n}$ have been preserved under the thinning. We remark that by definition of $T_{i,j}$ we have $\ell_{i-k_n}(T_{i-k_n,j+1}-) = j$, but as it will become clear in the sequel, it will be convenient to write $X^n_{i,j}$ in terms of this more complicated expression.

Finally, we define the processes $\ell^n_i$, $1\le i\le n$, by setting for $t \in [T_{i-k_n,j},T_{i-k_n,j+1})$,
\begin{align}\label{eq:definition-of-l-n-i}
  \ell^n_i(t) = X^n_{i,j}.
\end{align}
Note that the process $\ell^n_i$ jumps only by one and is allowed to jump only at times of jumps of $\ell_{i-k_n}$. Whether the $j$-th jump of $\ell_i$ becomes a jump of $\ell^n_i$ depends on the auxilliary variable $U_{i-k_n,j}$ as well as on the past of the process $\ell^n_i$.

Since for a while we are going to work with fixed $i$ and $n$, to simplify the notation we will temporarily write $T_j$ instead of $T_{i-k_n,j}$ and $U_j$ instead of $U_{i-k_n,j}$.

We will consider two filtrations. The first one is $\mathcal{G} = (\mathcal{G}_t)_{t\ge 0}$, where $\mathcal{G}_t$ is the $\sigma$-field generated by all events of the form
$A = \{\ell_i(t_1) = l_1,\ldots,\ell_i(t_m) = l_m, \ell_i(t) = l, U_1\in B_1,\ldots,U_l\in B_l\}$, where $m,l_1,\ldots,l_m,l\in \N$, $0\le t_1<\ldots t_m < t$ and $B_1,\ldots,B_l$ are Borel subsets of the real line. The second one is $\widetilde{\mathcal{G}} = (\widetilde{\mathcal{G}}_t)_{t\ge 0}$, where $\widetilde{\mathcal{G}}_t$ is the $\sigma$-field generated by all events of the form
$A = \{\ell_i(t_i) = l_1,\ldots,\ell_i(t_m) = l_m, \ell_i(t) = l, U_1\in B_1,\ldots,U_{l-1}\in B_{l-1}\}$, where $m,l_1,\ldots,l_m,l\in \N$, $0\le t_1<\ldots t_m < t$ and $B_1,\ldots,B_{l-1}$ are Borel subsets of the real line. Thus the difference between the two filtration is in the `information' about the value of $U_l$.

Note that $\ell_i^n$ is $\mathcal{G}$-adapted. We will first show that $\ell^n_i$ are counting processes with intensities $\lambda_i(t) = p_i(\ell^n_i(t),t)$ with respect to $\mathcal{G}$ (the relevant notions are recalled briefly in the Appendix for the reader's convenience).

Indeed, consider $0 \leq t < s$ and any $A\in \mathcal{G}_t$. Note that $A\cap \{t < T_{j} \le s\}$ is $\widetilde{\mathcal{G}}_{T_j}$ measurable for any $j \geq 0$. Moreover $U_j$ is independent of $\widetilde{\mathcal{G}}_{T_j}$. Observe that (crucially) $p_i(j,t) \le q(l,t)$ for any $j,l \in \N$ such that $j \le l$. In particular, we have $\frac{p_i(\ell^n_i(T_j-),T_{j})}{q(\ell_{i-k_n}(T_{j}-),T_{j})} \le 1$ and we can thus write
\begin{align*}
  \E \ell^n_i(s)\id_{A} - \E \ell^n_i(t)\id_{A}  & = \sum_{j=1}^\infty \E \Big(\id_{A\cap \{t < T_{j} \le s\}} \id\Big\{U_{j} \le \frac{p_i(\ell^n_i(T_j-),T_{j})}{q(\ell_{i-k_n}(T_{j}-),T_{j})}\Big\}\Big)\\
   & = \sum_{j=1}^\infty \E \Big(\id_{A\cap \{t < T_{j} \le s\}}\E\Big( \id\Big\{U_{j} \le \frac{p_i(\ell^n_i(T_j-),T_{j})}{q(\ell_{i-k_n}(T_{j}-),T_{j})}\Big\}\Big|\widetilde{\mathcal{G}}_{T_j}\Big)\\
& = \sum_{j=1}^\infty \E \Big(\id_{A\cap \{t < T_{j} \le s\}} \frac{p_i(\ell^n_i(T_j-),T_{j})}{q(\ell_{i-k_n}(T_{j}-),T_{j})}\Big) \\
&= \E \id_{A}\int_{(t,s]} \frac{p_i(\ell^n_i(u-),u)}{q(\ell_{i-k_n}(u-),u)}d\ell_{i - k_n}(u),
\end{align*}
where we used independence of $U_j$ and $\widetilde{\mathcal{G}}_{T_j}$ as well as the fact that $\ell^n_i(T_j-)$ is $\widetilde{\mathcal{G}}_{T_j}$-measurable.

The process $M_t = \ell_i(t) - \int_0^t q(\ell_i(u),u)du$ is a martingale with respect to the natural filtration of $\ell_i$. Using the $\pi-\lambda$ system theorem and independence of $U_i$'s from $\ell_i$ it is easy to show that it is also a martingale with respect to $\mathcal{G}$. Moreover, the process $Z:= (p_i(\ell^n_i(u-),u)/q(\ell_{i-k_n}(u-),u))_{u\ge 0}$ is $\mathcal{G}$-adapted and left-continuous, and thus $\mathcal{G}$-predictable. Thus, $\int_0^t Z_u dM_u$ is a $\mathcal{G}$-martingale and we can further write
\begin{displaymath}
  \E \ell^n_i(s)\id_{A} - \E \ell^n_i(t)\id_{A} = \E \id_{A}\int_t^s \frac{p_i(\ell^n_i(u-),u)}{q(\ell_{i-k_n}(u-),u)} q(\ell_{i - k_n}(u),u) du
  = \E \id_{A}\int_t^s p_i(\ell^n_i(u),u)du,
\end{displaymath}
where we used that $\ell_i$ has only countably many jumps on $[0,1)$ (in fact finitely many on any compact interval). This shows that $\ell^n_i(t) - \int_0^t p_i(\ell^n_i(u),u)du$ is a $\mathcal{G}$-martingale and so $\ell^n_i(t)$ is indeed a counting process with intensity $\lambda_i(t) = p_i(\ell^n_i(t),t)$. Thus, by Lemma \ref{lm:from-intensity-to-markov-rate} from the Appendix, $\ell^n_i$ is a Markov process on $\{0,\ldots,i-1\}$ with jump rates $p_i$. In particular, the  permutation process $\sigma^n_t$ defined by the requirement that $\ell^n_i$, $i\in [n]$, be its left inversion counts is a birth Mallows process.

We will use the process $\sigma^n$ to define $\Sigma^n$ coupled with $\Sigma$. Set
\begin{displaymath}
  \Sigma^n_t(i) = \sigma^n_t(k_n+i) - k_n
\end{displaymath}
for $i = -k_n+1,\ldots,n-k_n$ and $\Sigma^n_t(i) = i$ otherwise.

We will show that for any $N \in \N$ there exists a random variable $M_N < \infty$ such that, with probability one, for $n \ge M_N$, all $i \in \{-N,\ldots,N\}$ and $t \le T$ we have the equality $\Sigma^n_t(i) = \Sigma_t(i)$.

Observe first that $\Sigma^n_t$ is a balanced permutation of $\Z$ with left inversion counts $\ell_i(\Sigma^n_t) = \ell_{i+k_n}^n$ for $i \in \{-k_n+1,\ldots,n-k_n\}$ and 0 otherwise.
We have
\begin{displaymath}
\Sigma^n_t(i) = i + r_i(\Sigma^n_t) - \ell_i(\Sigma^n_t).
\end{displaymath}
where $r_i(\Sigma^n_t)$ are the right inversion counts, which can be obtained from $\ell_i(\Sigma_n^t)$ using formulas \eqref{eq:recursive-formula-ljt} and \eqref{eq:recursive-formula-r}.

Therefore, it suffices to show that for any $i \in \Z$, with probability one for $n$ large enough and all $t \in [0,T]$ we have $\ell_i(\Sigma^n_t) = \ell_i(t)$ and $r_i(\Sigma^n_t) = r_i(t)$.

In what follows we will work with a fixed element of the probability space, denoted by $\omega$. All the quantities appearing in the argument may depend on $\omega$. Let $i_1 \ge i$ be such that $\ell_{i_1}(T) = 0$. Recall that $\ell_i(\Sigma^n_t)$ equals either to zero or to $\ell^n_{i+k_n}(t)$, and so $\ell_i(\Sigma^n_t) \le \ell_i(t) \le \ell_i(T)$. Using this inequality, similarly as in the proof of Proposition \ref{prop:construction-of-limiting-sigma}, we show that there exists an integer $i_2 \geq i_1$ and a function $f=f_{T,i, \omega} \colon \N^{i_2-i+2} \to \N$ such that $r_i(\Sigma^n_t) = f(\ell_i(\Sigma^n_t),\ldots,\ell_{i_2+1}(\Sigma^n_t))$ and $r_i(t) = f(\ell_i(t),\ldots,\ell_{i_2+1}(t))$.
Therefore, if $\ell_m(\Sigma^n_t) = \ell_m(t)$ for all $m \in \{i,\ldots,i_2+1\}$ and $t \le T$, then also $r_i(\Sigma^n_t) = r_i(t)$ and as a consequence $\Sigma^n_t(i) = \Sigma_t(i)$ for $t \le T$.

Note that since $k_n, n-k_n \to \infty$, for $n$ large enough and all $m \in \{i, \ldots, i_2+1\}$ we have $\ell_m(\Sigma^n_t) = \ell^n_{m+k_n}(t)$. Thus, recalling \eqref{eq:definition-of-l-n-i},
\begin{align}\label{eq:ells-to-exes}
  \ell_m(\Sigma^n_t) = X^n_{m+k_n,j} \textrm{ for } t \in [T_{m,j},T_{m,j+1}), \; j=0,1,\ldots,
\end{align}
where $X^n_{i,j}$ are defined in \eqref{eq:definition-of-X}.

Let $A = \max \{\ell_{m}(T)\colon m \in \{i,\ldots,i_2+1\}\}$ and $U = \max\{U_{m,j}\colon m \in \{i,\ldots,i_2+1\}, j \in [A]\} < 1$. Observe that for any $j \in \{0,\ldots,A\}$ and $m \in \{i,\ldots,i_2+1\}$ the jump rate
\begin{displaymath}
  p_{m+k_n}(j,t) = \frac{1}{1-t} \left( j+1 - \frac{(m+k_n) t^{m+k_n-j-1}(t^{j+1} - 1)}{t^{m+k_n} - 1} \right)
\end{displaymath}
converges to
\begin{displaymath}
  q(j,t) = \frac{j+1}{1-t}
\end{displaymath}
as $n \to \infty$, uniformly over $t \in [0,T] \subseteq [0,1)$.

In particular, there exists $M=M_{T, i, \omega}$ such that for $n > M$, $m \in \{i,\ldots,i_2+1\}$, all $t \in [0,T]$ and $j \le A$, \eqref{eq:ells-to-exes} holds and we have
\begin{displaymath}
  \frac{p_{m+k_n}(j,t)}{q(j,t)} > U.
\end{displaymath}

When combined with \eqref{eq:definition-of-X} and easy induction, this shows that all the jumps of $\ell_m$ in $[0,T]$ will be preserved when passing to $\ell^n_{m+k_n}$, i.e., $X^n_{m+k_n,j} = \ell_m(T_{m,j})$ for all $m \in \{i,\ldots,i_2 + 1\}$ and $j \in \{0,\ldots, T_{m,\ell_m(T)}\}$. By \eqref{eq:ells-to-exes} this implies that for $n > M$ indeed $\ell_m(\Sigma^n_t) = \ell_m(t)$ for all $m \in \{i,\ldots,i_2+1\}$ and $t \in [0,T]$, and thus $\Sigma^n_t(i) = \Sigma_t(i)$. This ends the proof of the theorem.
\end{proof}
\end{subsection}
\end{section}

\begin{section}{Appendix}\label{sec:appendix}

\paragraph{Skorokhod topology on permutation paths.} Let us first recall the basic definitions related to the Skorokhod topology on $D([0,1),\GSs^{bal})$. We will follow the exposition in \cite[Chapter 3.5]{MR0838085}, where the case $D([0,\infty),E)$ is considered for an arbitrary metric space $(E,r)$ (we refer also to \cite{MR1943877} and \cite{MR1700749}). We will rely on their definitions, with an identification of $[0,1)$ with $[0,\infty)$ via a homeomorphism $\tau$. For concreteness let us take $\tau(t) = -\log(1-t)$, which is the same function as used in Subsection \ref{subsec:construction-of-limit}.

To metrize the topology of pointwise convergence on $\GSs^{bal}$, let us endow it with the metric $\rho(\sigma,\eta) = \frac{1}{3}\sum_{n\in \Z} 2^{-|n|} \id_{ \{\sigma(i) \neq \eta(i)\}}$ (the factor $\frac{1}{3}$ is introduced just to make $r$ bounded by one, which allows to simplify the formulas from \cite[Chapter 3.5]{MR0838085}). Let $\Lambda$ be the collection of strictly increasing, Lipschitz functions $\lambda$, mapping $[0,\infty)$ onto itself, such that
\begin{displaymath}
  \gamma(\lambda) := \sup_{s > t \ge 0}\Big|\log \frac{\lambda(s) - \lambda(t)}{s-t}\Big| < \infty.
\end{displaymath}

We define the space $\DD = D([0,\infty),\GSs^{bal})$ as the space of all c\`{a}dl\`{a}g functions on $[0,1)$ with values in $\GSs^{bal}$.

For $x,y \in \DD$, $\lambda \in \Lambda$ and $u \ge 0$ we set
\begin{displaymath}
  d(x,y,\lambda,u) = \sup_{t \in [0,\infty)} \rho(\widetilde{x}(t\wedge u), \widetilde{y}(\lambda(t)\wedge u)),
\end{displaymath}
where $\widetilde{x} = x\circ \tau$, $\widetilde{y} = y\circ \tau$. Finally, we define
\begin{displaymath}
  d(x,y) = \inf_{\lambda \in \Lambda} \max \Big(\gamma(\lambda), \int_0^\infty e^{-u}d(x,y,\lambda,u)du\Big).
\end{displaymath}
Then $d$ is a metric and the space $(\DD, d)$ is separable and complete.

Let us comment that if $x_n \to x$ in $\DD$, then for any $T \in [0,1)$ such that $x$ is continuous at $T$, we have $x_n\vert_{[0,T]} \to x\vert_{[0,T]}$ in the usual Skorokhod topology $J_1$ on the space $D([0,T],\GSs^{bal})$ of c\`{a}dl\`{a}g paths on a compact interval (see, e.g., \cite{MR1700749}). In the other direction, if for every $T \in [0,1)$ it holds that $x_n\vert_{[0,T]} \to x\vert_{[0,T]}$ in $D([0,T],\GSs^{bal})$, then $x_n \to x$ in $\DD$.

\paragraph{Counting processes and intensities.} Recall that $(N_t)_{t \geq 0}$ is a \emph{counting process} if it is a nondecreasing, integer-valued c\`{a}dl\`{a}g stochastic process with jumps equal to one. If $N$ is $\mathcal{F}_t$-adapted, then a nonnegative process $(\lambda_t)_{t \geq 0}$ is an \emph{intensity} of $N$ if it is $\mathcal{F}_t$-progressively measurable, $\int\limits_{0}^{t} \lambda_s \, ds < \infty$ for every $t \geq 0$ and the process $N_t - \int\limits_{0}^{t} \lambda_s \, ds$ is an $\mathcal{F}_t$-martingale.

The following lemma, used in the proof of Proposition \ref{prop:evolution-of-inversions}, follows easily from \cite[Corollary 1]{MR2245573}.
\begin{lemma}\label{le:concentration-counting-process}
Let $(N_t)_{t\ge 0}$ be a counting process with intensity $(\lambda_t)_{t\ge 0}$ bounded by $a$. Set $Z_t = N_t - \int_0^t \lambda_u \, du$. Then for any $s,u>0$,
\begin{displaymath}
  \Pp(\sup_{t\le s} |Z_t| \ge u) \le 2\exp\Big(- \frac{u^2}{4as + 2u/3}\Big).
\end{displaymath}
\end{lemma}

The next lemma, used in the proof of Theorem \ref{th:main-theorem-local}, gives a criterion for showing when a counting process is in fact a Markov process.

\begin{lemma}\label{lm:from-intensity-to-markov-rate}
Let $(N_t)_{t \geq 0}$ be a counting process with intensity $(\lambda_{t})_{t \geq 0}$. If $\lambda_t = p(N_t, t)$, where $p(\cdot,\cdot)$ is a bounded measurable function continuous in $t$, then $N_t$ is a Markov birth process with birth rate given by $p$.
\end{lemma}

This rather intuitive statement can be proved by considering, for a given $s > 0$ and any $t \geq s$, the conditional transition probabilities of the form $\Pp(\{N_t = k\} | \{N_s = j\} \cap A)$, where $A \in \mathcal{F}_{s-}$. One then shows that they satisfy the same Chapman--Kolmogorov forward equations as the function $p(\cdot,\cdot)$ and since the solution is unique, it follows that $(N_t)_{t \geq 0}$ is in fact a Markov process. A related statement can be found in \cite[Proposition 3.3]{spreij} (see also \cite{jacobsen}).

\paragraph{Auxilliary results for the global limit.} Here we collect several technical statements and proofs which are used in Section \ref{sec:global-limit} along the way to the proof of Theorem \ref{th:main-theorem-global}.

We start with the proof of Corollary \ref{cor:ldp}.

\begin{proof}[Proof of Corollary \ref{cor:ldp}]
In what follows we fix $T > 0$ and suppress the dependence of the constants on $T$ in the notation.

Let $d_w$ be any distance on $\MM([0,1]^2)$ metrizing the weak convergence and recall the definition \eqref{eq:box-topology} of the box metric $d_{\Box}$.
By $B_w(\mu,\varepsilon)$ and $B_\Box(\mu,\varepsilon)$ we will denote the open balls with radius $\varepsilon$ around $\mu$ in $\MM([0,1]^2)$ with respect to the metrics $d_w$, $d_\Box$.

Let $\mu_{\beta}$ be the measure defined in Theorem \ref{th:starr-lln} with density $\rho_\beta$ given by \eqref{eq:definition-of-rho-beta}. Note that for any $x,y \in [0,1]$ the function $\beta \mapsto \rho_\beta(x,y)$ is continuous. Thus, by Scheff\'{e}'s lemma the function $\beta \mapsto \mu_\beta$ is continuous in the total variation distance.

For a permutation $\sigma \in \Ss_n$ let $\widehat{\mu}_{\sigma}$ be the permuton with density $f_\sigma(x,y) = n \id_{\sigma(\lceil nx\rceil) = \lceil ny\rceil}$. One easily checks that $d_\Box(\widehat{\mu}_\sigma,\mu_\sigma) \le \frac{2}{n}$. Using the fact that the box topology and the weak topology agree on the set of permutons $\PP$, together with continuity of the map $\beta \mapsto \mu_\beta$, by a compactness argument we conclude that for any $\varepsilon > 0$ there exists $\varepsilon' > 0$ such that for any $\beta \in [0,T]$ and $\sigma \in \Ss_n$, if $\mu_\sigma \in B_w(\mu_\beta,\varepsilon')$ then $\mu_\sigma \in B_\Box(\mu_\beta,\varepsilon)$.

Thus for fixed $\varepsilon > 0$ we can apply the upper bound of Theorem \ref{th:starr-walters-ldp} with $A= B_w(\mu_\beta,\varepsilon')^{c}$, which gives
\[
\limsup\limits_{n \to \infty} \frac{1}{n} \log \Pp_{n, \beta} \left( d_{\Box}\left(\mu_{\sigma^{n}}, \mu_{\beta} \right) \geq \varepsilon \right) \leq - \inf\limits_{\mu \in B(\mu_\beta,\varepsilon')^c} \I_{\beta}(\mu) =: -c_\beta.
\]
By Proposition \ref{prop:starr-walters-minimizer} $\mu_\beta$ is the unique minimizer of $\I_\beta$ and $\I_\beta(\mu_\beta) = 0$, so using lower-semicontinuity of $\I_\beta$ together with compactness of its sublevel sets, we get $c_\beta > 0$. Note that for any measure $\mu$, the `energy' $\EE(\mu)$ in the definition of $\I_\beta(\mu)$ is bounded by $1/2$, moreover the function $\beta\mapsto p(\beta)$ is continuous. This implies that the family of functions  $\{\beta \mapsto \I_\beta(\mu)\colon S(\mu|\lambda^{\otimes 2}) < \infty\}$ is uniformly equicontinuous on $[0,T]$. Thus, another compactness argument  (using again the continuity of the map $\beta \mapsto \mu_\beta$) implies that $\inf_{\beta \in [0,T]} c_\beta > 0$.
 
Using the definition of the box metric and of the measures $\mu_{\sigma_n}$, $\mu_\beta$, we thus obtain that for each $\varepsilon > 0$, there exists $b_{\varepsilon} > 0$ such that for any $\beta \in [0,T]$ there exists $n_{\beta,\varepsilon} \geq 1$ with the property that for all $n \ge n_{\beta,\varepsilon}$
\begin{align}\label{eq:concentration-single-beta}
\Pp_{n, \beta} \left( \sup\limits_{R} \left| \Delta_{R}\left(\sigma^{n}\right) - \int\limits_{R} \rho_{\beta}(x,y) \, dx \, dy \right| \geq \varepsilon \right) \leq  e^{-b_{\varepsilon} n}.
\end{align}
Our goal is now to argue that one may choose the constants $n_{\beta,\varepsilon}$ to be uniformly bounded over $\beta \in [0,T]$. The proof will be then concluded by adjusting the constants.

Let us fix $\varepsilon > 0$. By continuity of the function $\beta \mapsto \mu_\beta$ in the total variation distance, there exists $\delta > 0$ such that for any $\beta,\beta' \in [0,T]$ less than $\delta$ apart,
\[
\sup_R \Big| \int\limits_{R} \rho_{\beta}(x,y) \, dx \, dy - \int\limits_{R} \rho_{\beta'}(x,y) \, dx \, dy\Big| < \frac{\varepsilon}{2}.
\]
Let $0 = \beta_0 < \beta_1 <\ldots< \beta_N = T$ be a finite collection of points such that $\beta_{j+1} - \beta_j < \min(\delta, \frac{b_{\varepsilon/2}}{2})$. Consider now any $\beta \in [\beta_j,\beta_{j+1}]$, denote the left-hand side of \eqref{eq:concentration-single-beta} by $p_{n,\beta}(\varepsilon)$ and observe that
\begin{align}\label{eq:approximation-from-left}
p_{n,\beta}(\varepsilon) \le \Pp_{n, \beta} \left( \sup\limits_{R} \left| \Delta_{R}\left(\sigma^{n}\right) - \int\limits_{R} \rho_{\beta_j}(x,y) \, dx \, dy \right| \geq \frac{\varepsilon}{2} \right).
\end{align}

Recall the definition \eqref{eq:definition-of-mallows} of the Mallows distribution $\pi_{n,q}$. For any $B \subseteq \Ss_n$, we have 
\begin{multline*}
  \Pp_{n,\beta}(\sigma^n \in B) = \int\limits_{B} e^{(\beta - \beta_j) \frac{\Inv(\sigma)}{n}} \cdot \frac{\Zz_{n,\exp(\beta_j/n)}}{\Zz_{n,\exp(\beta/n)}} d\pi_{n,\exp(\beta_j/n)}(\sigma)\\
  \le e^{(\beta - \beta_j)\frac{n-1}{2}} \frac{\Zz_{n,\exp(\beta_j/n)}}{\Zz_{n,\exp(\beta/n)}} \Pp_{n,\beta_j}(\sigma^n \in B)
  \le e^{(\beta - \beta_j)\frac{n-1}{2}} \Pp_{n,\beta_j}(\sigma^n \in B),
\end{multline*}
where in the first inequality we used $\Inv(\sigma) \leq \frac{n(n-1)}{2}$ and in the second one monotonicity of $\Zz_{n,\exp(\beta/n)}$ with respect to $\beta$.
Taking into account \eqref{eq:approximation-from-left} and \eqref{eq:concentration-single-beta} we thus obtain that for $n \ge \max\limits_{j<N} n_{\beta_j,\varepsilon/2}$,
\begin{displaymath}
  \sup_{\beta \in [0,T]} p_{n,\beta}(\varepsilon) \le  \max_{j< N} \Big( e^{(\beta - \beta_j)\frac{n-1}{2}} p_{n,\beta_j}(\varepsilon/2)\Big) \le  e^{- (b_{\varepsilon/2} / 2) n},
\end{displaymath}
which ends the proof.
\end{proof}

\begin{lemma}\label{le:equicontinuity-of-z}
For any $T > 0$, the family of functions $\{z_{x,a}\colon x,a \in [0,1]\}$, where $z_{x,a}(t)$ is given by \eqref{eq:formula-for-global-limit}, is uniformly equicontinuous on the interval $[0,T]$.
\end{lemma}

\begin{proof}
We can rewrite the formula \eqref{eq:formula-for-global-limit} as
\[
z_{x,a}(t) = \frac{1}{t} \log \left( 1 + a \frac{e^t - 1}{e^{xt}(1-a) + a} \right).
\]
Since the denominator under the logarithm is at least one for $x,a \in [0,1]$, we have
\[
z_{x,a}(t) = \frac{1}{t} \cdot a\frac{e^t - 1}{e^{xt}(1-a) + a} + o\left(\frac{1}{t}(e^t - 1)\right)
\]
as $t \to 0$, which implies that $z_{x,a}(t) \to a$, uniformly over $a,x \in [0,1]$. Moreover, the function $(a,x,t) \mapsto z_{x,a}(t)$ is clearly continuous and hence uniformly continuous on $[0,1]^2 \times [\varepsilon,T]$ for any $0<\varepsilon < T <\infty$.
\end{proof}

\begin{lemma}\label{le:Lipschitz-condition}
For each $T > 0$ there exists a constant $C = C_T > 0$ such that for all $t \in [0,T]$ the function $(x,y) \mapsto \lambda_x(y,t)$, with $\lambda_x(y,t)$ defined by \eqref{eq:rescaled-rates}, is $C$-Lipschitz on $[0,1]^2$.
\end{lemma}

\begin{proof}
The case $t = 0$ is straightforward, let us now focus on $t > 0$.
Using that for some constant $D$ and all $a\in [-T,T]$ we have $|a|/D \le |e^{a} - 1| \le D|a|$ and $0\le e^{a} - a - 1 \le Da^2$, we obtain that for $t > 0$, $x \in (0,1]$ and $y \in [0,1]$,
\begin{multline*}
  \Big|\frac{\partial}{\partial x} \lambda_x(y,t) \Big|= \Big| \frac{(1-e^{-ty})}{t} \cdot \frac{(e^{tx} + xte^{tx})(e^{tx}-1) - txe^{2tx}}{(e^{tx} - 1)^2}\Big| = \Big|\frac{(1-e^{-ty})}{t} \cdot \frac{e^{tx}(e^{tx} - tx -1)}{(e^{tx}-1)^2}\Big|\\
  \le e^{tx} D^4 y \le D^4 e^T.
\end{multline*}
Similarly, for $t > 0$, $x \in (0,1]$, $y \in [0,1]$,
\begin{multline*}
  \Big|\frac{\partial}{\partial y} \lambda_x(y,t) \Big| = \frac{1}{t} \Big|-1 + \frac{txe^{tx}}{e^{tx}-1}e^{-ty}\Big| = \Big| \frac{1 - e^{tx} + txe^{tx}e^{-ty}}{t(e^{tx}-1)} \Big|\\
  \le \Big|\frac{(e^{tx} - tx - 1) -  tx(e^{tx}-1) - txe^{tx}(e^{-ty} - 1)}{t(e^{tx}-1)}\Big| \le D^2(2x+ye^{tx}) \le 3 D^2 e^T.
\end{multline*}
Thus for every $t \in (0,T]$, the function $(x,y) \mapsto \lambda_{x}(y,t)$ is $C$-Lipschitz on $(0,1]\times [0,1]$ with some constant $C$ depending only on $T$. Since for fixed $y$ and $t$ the function $(x,y) \mapsto \lambda_x(y,t)$ is continuous at $x = 0$, this concludes the proof.
\end{proof}

\begin{lemma}\label{lm:density-lower-bound}
For $\beta \in [0,T]$ the density $\rho_{\beta}(x,y)$ given by \eqref{eq:definition-of-rho-beta} is bounded from below by some $A = A_T > 0$ depending only on $T$.
\end{lemma}

\begin{proof}
Let $u = \frac{x-y}{2}$, $v = \frac{x+y-1}{2}$, so that $(x,y) \in [0,1]^2$ corresponds to $(u,v) \in [-\frac{1}{2}, \frac{1}{2}]^2$, $|u| + |v| \leq 1/2$. Let us further write
\[
\rho_{\beta}(x,y) = \widetilde{\rho}_{\beta}(u,v) = \frac{(\beta / 2) \sinh(\beta / 2)}{\left( e^{-\beta / 4} \cosh \beta u - e^{\beta / 4} \cosh \beta v  \right)^2}.
\]
By differentiating the function in the denominator one easily checks that it cannot attain its maximum in the interior of the set $\{(u,v) \in [-\frac{1}{2}, \frac{1}{2}]^2 \, : \, |u| + |v| \leq 1/2\}$. Likewise, an elementary calculation shows that on the boundary its maximum may be attained only at one of the points $(u,v) \in \left\{ (\frac{1}{2},0), (0,\frac{1}{2}), (-\frac{1}{2},0), (0, -\frac{1}{2}) \right\}$. Thus it is enough to verify that $\widetilde{\rho}_{\beta}$ is bounded away from $0$, uniformly in $\beta$, at these points. Since the calculation is similar in all cases, let us only check, e.g., $(u,v) = (\frac{1}{2},0)$. We have
\begin{align*}
\widetilde{\rho}_{\beta}\left(\frac{1}{2},0\right) & =  \frac{(\beta / 2) \sinh(\beta / 2)}{\left( e^{-\beta / 4} \cosh \frac{\beta}{2} - e^{\beta / 4}  \right)^2} = \frac{(\beta / 2) \sinh(\beta / 2)}{ e^{-\beta / 2} \cosh^2 \frac{\beta}{2} + e^{\beta / 2} - 2\cosh \frac{\beta}{2}} = \\
& = \frac{\left(\frac{\beta}{2}\right) \left( \frac{e^{\beta/2} - e^{-\beta/2}}{2}\right)}{ e^{-\beta / 2} \frac{e^\beta + e^{-\beta} + 2}{4} - e^{-\beta/2}} = \frac{\beta (e^\beta - 1)}{ e^\beta + e^{-\beta} - 2},
\end{align*}
which is easily seen to be bounded away from $0$ uniformly in $\beta \in [0,T]$.
\end{proof}
\end{section}

\bibliography{bibliography}{}
\bibliographystyle{amsalpha}

\paragraph{E-mails:}
R. Adamczak: \href{mailto:r.adamczak@mimuw.edu.pl}{r.adamczak@mimuw.edu.pl}, M. Kotowski: \href{mailto:michal.kotowski@mimuw.edu.pl}{michal.kotowski@mimuw.edu.pl}

\end{document}